\documentclass{article}

\usepackage[UKenglish]{babel}
\usepackage{amsmath, amssymb, amsthm}
\usepackage{lmodern}
\usepackage{vmargin}
\usepackage{tikz}
\usetikzlibrary{matrix}
\usepackage{stmaryrd}
\usepackage[shortlabels]{enumitem}

\newcounter{maincounter}[section]

\theoremstyle{plain}
\newtheorem{notation}[maincounter]{Notation}
\newtheorem*{notation*}{Notation}
\newtheorem*{question*}{Question}
\newtheorem*{theorem*}{Theorem}
\newtheorem{proposition}[maincounter]{Proposition}
\newtheorem{corollary}[maincounter]{Corollary}
\newtheorem{remark}[maincounter]{Remark}
\newtheorem*{fact*}{Fact}
\newtheorem*{TQT*}{Timmesfeld's Quadratic Theorem}
\newtheorem*{logarithm*}{Lie-ring analogue}

\theoremstyle{definition}
\newtheorem{definition}[maincounter]{Definition}
\theoremstyle{remark}
\newtheorem{claim}{Claim}[maincounter]

\makeatletter
\newenvironment{proofclaim}
{\par\pushQED{\qed}
\normalfont \topsep6\p@\@plus6\p@\relax\trivlist
\item[\hskip\labelsep
\emph{Proof of Claim.}]
\ignorespaces

}
{\popQED\endtrivlist\@endpefalse}
\makeatother

\newcommand{\simple}{-}
\newcommand{\doubleleq}{\Leftarrow}
\newcommand{\doublegeq}{\Rightarrow}

\newcommand{\tripleleq}{\Lleftarrow}

\newcommand{\<}{\langle}
\renewcommand{\>}{\rangle}

\newcommand{\cl}{\mathrm{cl}}

\newcommand{\K}{\mathbb{K}}
\newcommand{\R}{\mathbb{R}}
\newcommand{\Z}{\mathbb{Z}}
\newcommand{\bG}{\mathbb{G}}
\newcommand{\bSL}{\mathbb{SL}}

\newcommand{\cG}{\mathcal{G}}
\newcommand{\cu}{y}
\newcommand{\cU}{\mathcal{U}}
\newcommand{\cT}{\mathcal{T}}

\newcommand{\fg}{\mathfrak{g}}
\newcommand{\fu}{\mathfrak{u}}
\newcommand{\ft}{\mathfrak{t}}

\newcommand{\Triv}{Z}
\newcommand{\Lie}{\operatorname{Lie}}

\newcommand{\Id}{\operatorname{Id}}

\newcommand{\SL}{\operatorname{SL}}
\newcommand{\SU}{\operatorname{SU}}

\newcommand{\pSL}{\operatorname{(P)SL}}

\renewcommand{\sl}{\mathfrak{sl}}
\newcommand{\Ann}{\operatorname{Ann}}
\newcommand{\End}{\operatorname{End}}

\newcommand{\Nat}{\operatorname{Nat}}

\title{Locally quadratic modules and minuscule representations}
\author{Adrien Deloro}

\begin{document}

\maketitle

\renewcommand{\>}{\rangle}
\renewcommand{\sl}{\mathfrak{sl}}
\renewcommand{\th}{^{\rm th}}
\renewcommand{\d}{\partial}

\begin{flushright}
\em Being natural is simply a pose, and the most irritating pose I know.
\end{flushright}

\begin{abstract}
We give a new proof of a theorem by Timmesfeld showing that for simple algebraic groups, abstract modules where all roots act quadratically are direct sums of minuscule representations.
\end{abstract}


In order to handle groups and Lie rings in a single statement one needs a bit of notation.

\begin{notation*}\
\begin{itemize}
\item
If $G$ is a group and $V$ is a $\Z[G]$-module, let $\Triv_V(G) = C_V(G) = \{v \in V: \forall g\in G, g \cdot v = v\}$ and $\lfloor G, V\rfloor = [G, V] = \< g\cdot v - v: (g, v) \in G\times V\>$;
\item
if $\fg$ is a Lie ring and $V$ is a $\Z[\fg]$-module, let $\Triv_V(\fg) = \Ann_V(\fg) = \{v \in V: \forall z \in \fg, z\cdot v = 0\}$ and $\lfloor \fg, V\rfloor = \fg \cdot V = \<z \cdot v: (z, v) \in \fg \times V\>$.
\end{itemize}
\end{notation*}

\begin{theorem*}[{also in \cite{TComplete}}]
Let $\K$ be a field of characteristic $\neq 2$ with more than three elements and $\bG$ be one of the simple algebraic groups (of classical or exceptional type; untwisted). Let $G = \bG_\K$ be the abstract group of $\K$-points of the functor $\bG$ and $\fg = (\Lie \bG)_\K$ be the abstract Lie \emph{ring} of $\K$-points of the functor $\Lie \bG$. Let $\cG$ be either $G$ or $\fg$ and $V$ be a $\Z[\cG]$-module.

Suppose that all roots act quadratically.
Then $V = \Triv_V(\cG) \oplus \lfloor \cG, V\rfloor$ and $\lfloor \cG, V\rfloor$ can be equipped with a $\K$-vector space structure making it isomorphic to a direct sum of minuscule representations of $\cG$ as a $\K[\cG]$-module.
\end{theorem*}

The introduction will motivate the statement (\S\S\ref{s:quadraticity} and \ref{s:minuscule}) and also explain what made us give a new proof -- which we believe is completely natural (\S\ref{s:Musset}). Further comments are made in \S\ref{s:remarks}.
The argument itself is in \S\ref{S:proof}.

\section{Introduction}\label{S:Introduction}

The article studies some representations of the simple algebraic groups as abstract group modules
. This amounts to doing representation theory on purely group-theoretic grounds.
%
The topic which originally attracted us is that of linear reconstruction: given an algebraic group $G$ over $\K$ seen as an abstract group and a $\Z[G]$-module $V$, try to retrieve a $\K$-linear structure on $V$ induced by the action of $G$.

\subsection{Quadratic Actions}\label{s:quadraticity}

A typical example of linear reconstruction is the following theorem which was proved in the mid-eighties by S. Smith and F.G. Timmesfeld, independently. $U$ stands for a unipotent subgroup of $\SL_2(\K)$, say the group of upper-triangular matrices with $1$ on the diagonal; quadraticity of the $G$-module $V$ means that $[U, U, V] = 0$ (which does not depend on the unipotent subgroup by conjugacy).

\begin{TQT*}[{\cite[Exercise 3.8.1 of chapter I]{Timmesfeld}; also \cite{Smith}}]
Let $\K$ be a field of characteristic $\neq 2$ with more than three elements, $G = \SL_2 (\K)$, and $V$ be a \emph{quadratic} $G$-module. 
Then $V = C_V(G) \oplus [G, V]$, and there exists a $\K$-vector space structure on $[G, V]$ making it isomorphic to a direct sum of copies of $\Nat \SL_2(\K)$ as a $\K[G]$-module.
\end{TQT*}

Since the field $\K$ is rather arbitrary, there are no character-theoretic nor Lie-theoretic methods available; $\SL_2(\K)$ is seen as an abstract group with no extra structure, and the proof is therefore by computation. One fixes generators and works with the so-called Steinberg relations for $\SL_2(\K)$.

The lack of Lie-theoretic information incidently suggests to ask the same question about $\sl_2(\K)$-modules. For the problem of linear reconstruction to make sense we view $\sl_2(\K)$ as a Lie \emph{ring}, viz. an abelian group with a bracket (forgetting the underlying vector space structure); an $\sl_2(\K)$-module need not be a vector space over $\K$.
We let $\fu_+$ and $\fu_-$ be the abelian subrings of upper-triangular (resp. lower-triangular) matrices with $0$ on the diagonal.
Quadraticity of the $\fg$-module $V$ now means that both $\fu_+^2\cdot V = \fu_-^2 \cdot V = 0$ (see \S\ref{s:remarks} for more on this two-sided assumption).

\begin{logarithm*}[\cite{TV-I}]
Let $\K$ be a field of characteristic $\neq 2$, $\fg = \sl_2(\K)$, and $V$ be a quadratic $\fg$-module.
Then $V = \Ann_V(\fg) \oplus \fg\cdot V$, and there exists a $\K$-vector space structure on $\fg\cdot V$ making it isomorphic to a direct sum of copies of $\Nat \sl_2(\K)$ as a $\K[\fg]$-module.
\end{logarithm*}

\subsection{Minuscule Representations}\label{s:minuscule}

By definition, the minuscule representations of a semisimple Lie algebra are its irreducible representations such that the action of the Weyl group on the set of weights is transitive; the latter condition is equivalent to: every root element acts with $x^2 = 0$ \cite[Chap. VIII, \S7.3, Propositions 6 and 7]{BLie78}. In the simple case, the list of minuscule weights can be determined from that of fundamental weights, and minuscule representations of the various simple Lie algebras are therefore known. They are as follows: all exterior powers of the natural representation for type $A_n$, the spin representation for type $B_n$, the natural representation for type $C_n$, the natural and the two half-spin representations for type $D_n$, two representations for type $E_6$, one for type $E_7$, none for types $E_8$, $F_4$, $G_2$. \cite[Chap. VIII, end of \S7.3]{BLie78}.

It is tempting to see Timmesfeld's Quadratic Theorem and its Lie-ring analogue as identification results for the unique minuscule representation of the algebraic group $\bSL_2$ among abstract $G$- or $\fg$-modules. And indeed, our result is the natural extension of Timmesfeld's quadratic theorem to the other simple algebraic groups (and to their Lie algebras, seen as Lie rings).



%
%



\subsection{Je suis venu trop tard dans un monde trop vieux}\label{s:Musset}

Only while typing our proof did we learn about the following.

\begin{fact*}[Timmesfeld, {\cite{TComplete}}]
Let $G$ be a finite Lie-type group over $GF(q)$, $q = p^n$, $p \neq 2$, different from $\SL_2(3)$, with Dynkin diagram $\Delta = \Delta(I)$ and let $V$ a $\Z_p G$-module, on which the root groups of $G$ act quadratically, i.e. $[V, A_r, A_r] = 0$ for all roots $r$ of the root system of $G$. Then $V = C_V(G) \oplus [V, G]$ and $[V, G]$ is the direct sum of irreducible $\Z_p G$-modules $V_j$ \emph{[the list of which is as expected and explicitly given]}.
\end{fact*}

Moreover, Timmesfeld observes that in the case of non-exceptional groups his proof extends to \emph{infinite} fields of characteristic not $2$. He also handles twisted groups (notably $\SU_n$), and we cannot do the twist.

Stumbling upon \cite{TComplete} was a blow to the author.
Our goal was to generalise Timmesfeld's partial and much earlier work \cite{TIdentification}. But \cite{TComplete} was submitted before we even started to think about the topic; by the time it was published, we could only treat $\SL_n(\K)$ and $\sl_n(\K)$. Only two years after its publication did we: complete our proof, type it down, and having finally done this, become aware of \cite{TComplete}.

Why make our own work public then? Covering Lie rings as well is certainly no sufficient reason.
But we believe that our proof will not lack interest since it is:
\begin{itemize}
\item
essentially different from Timmesfeld's -- having been developed independently as we explained, our method linearises without caring for what the resulting representation will be (something which can be determined afterwards, if necessary), while the philosophy of \cite{TComplete} involves explicit module identification;
\item
uniform, while \cite{TComplete} is a case division;
\item
entirely self-contained modulo the Quadratic Theorem and its Lie-ring analogue, while \cite{TComplete} requires non-trivial representation-theoretic information (\cite[end of \S1]{TComplete}; \cite[Lemma 2.6]{TComplete} is one crux of the argument); 
\item
effective since the linear structure is defined explicitly provided one has realised one root substructure of type $A_1$ and the global Weyl group;
\item
transparent -- in our opinion.
Our proof is about minuscule weights and transitivity of the Weyl group, which are the natural phenomena to investigate when one is talking about minuscule representations. The structure of the argument is explained at the beginning of \S\ref{S:proof}.
\end{itemize}


\subsection{Remarks and Questions}\label{s:remarks}

Our last introductory subsection consists of remarks on the statement and its proof, together with a few questions on possible extensions. The technical discussion here involves a number of pathologies and is not necessary in order to understand the proof in \S\ref{S:proof}.

\begin{itemize}
\item
Our interpretation of ``all roots act quadratically'' is that all root $\bSL_2$-substructures in one realisation of $\cG$ act quadratically. This was explained in \S\ref{s:quadraticity}. For $\alpha$ a root let $\cU_\alpha$ be the associated root substructure (i.e., subgroup or Lie subring).
\begin{enumerate}
\item
As stated the assumption means: for $\alpha \in \Phi$ (the root system), $\lfloor \cU_\alpha, \lfloor \cU_\alpha, V\rfloor \rfloor = 0$.
\item
One may try to restrict to positive roots: for $\alpha \in \Phi_+$ (positive roots), $\lfloor \cU_\alpha, \lfloor \cU_\alpha, V\rfloor \rfloor = 0$.
\item
One may try to restrict to simple roots: for $\alpha \in \Phi_s$ (simple roots), $\lfloor \cU_\alpha, \lfloor \cU_\alpha, V\rfloor \rfloor = 0$.
\item
One may try to restrict to root elements: for $\alpha \in \Phi$ (resp. $\Phi_+$, $\Phi_s$) and some $\cu \in \cU_\alpha$ not the identity, $\lfloor \cu_\alpha, \lfloor \cu_\alpha, V\rfloor \rfloor = 0$.
\end{enumerate}
These slight variations can have unexpected effects since we are dealing with an abstract group or Lie ring.
\begin{itemize}
\item
For instance, supposing that the positive Lie subring $\fu_+ \leq \sl_2(\K)$ acts quadratically does not fully guarantee that so does the negative Lie subring $\fu_-$.
In characteristic neither $2$ nor $3$ these turn out to be equivalent \cite[Variation 12]{TV-I} but in characteristic $3$ one can construct $\sl_2(\K)$-modules with $\fu_+^2 \cdot V = 0 \neq \fu_-^2 \cdot V$ \cite[\S4.3]{TV-I}. (For a more general discussion of the non-equality of nilpotence orders of generators of $\fu_+$ and $\fu_-$ in $\End(V)$, see \cite[\S\S3.2 and 3.3]{TV-II}.)
As a consequence, an assumption restricted to positive roots is too weak.
\item
The reader is now aware that in the case of the Lie ring, lifting the action of the Weyl group on roots to an action on the module is non-trivial.
The caveat extends to Lie rings not of type $A_1$. It is therefore not clear whether all simple roots of the same length must have similar actions on $V$ since conjugacy under the Weyl group may not be compatible with the action on the module. So an assumption restricted to simple roots and their opposites could be too weak as well (we did not look for a counter-example).
\item
Finally, it is the case that for an action of $\SL_2(\K)$ in characteristic neither $2$ nor $3$, $(u-1)^2 = 0$ for some element $u \in U\setminus\{1\}$ implies $[U, U, V] = 0$ \cite[Variation 7]{TV-I}; we do not know what happens in characteristic $3$ (bear in mind that the field can be infinite). For an action of $\sl_2(\K)$ it suffices to be in characteristic not $2$: $x^2\cdot V = 0$ for some $x \in \fu_+\setminus\{0\}$ does imply $\fu_+^2\cdot V = 0$ \cite[Variation 9]{TV-I} (which in characteristic $3$ does however not entail $\fu_-^2\cdot V = 0$ as we just said). Hence an assumption restricted to root elements is too weak.
\end{itemize}
There are two conclusions. First, in the case of the group and characteristic not $3$, it would be enough to suppose that one element in one root subgroup of each length is quadratic -- which makes an assumption on at most two elements. Now in the case of Lie rings, apparently minor changes in the hypothesis can give rise to pathologies we shall prefer not to discuss in the course of the argument: let us stick to the assumption that for any $\alpha \in \Phi$, $\lfloor \cU_\alpha, \lfloor \cU_\alpha, V\rfloor \rfloor = 0$.
\item
Timmesfeld \cite[Corollary]{TComplete} proves that in the case of the \emph{groups} of type $A-D-E$ over \emph{finite fields of characteristic not $2$}, it suffices that one element of $G$ acts quadratically for $V$ to be as above.

Here are a few limits to Timmesfeld's Corollary:
\begin{enumerate}
\item
as noted in \cite{TComplete}, one needs type $A-D-E$ in order to conjugate roots (for the other types, Timmesfeld discusses why an assumption on long roots is not enough);
\item
the group must be finite since the result uses the classification of so-called quadratic pairs (and therefore, if we are not mistaken, the characteristic $3$ analysis in \cite[Corollary]{TComplete} does require the group to be finite);
\item
$G$ must be a group and cannot be a Lie ring, since we noted that the statement fails for $\sl_2(\K)$ in characteristic $3$.
\end{enumerate}

Our alternative approach to Timmesfeld's Corollary, not using the classification of quadratic pairs, and valid for \emph{groups} of type $A-D-E$ over \emph{possibly infinite fields of characteristic neither $2$ nor $3$}, can be read from the previous remark: use \cite[Variation 7]{TV-I} to find a quadratic root subgroup, then conjugate roots.
\item
Our proof (and this is the one thing it has in common with Timmesfeld's) makes crucial use of the ``characteristic not $2$'' assumption as it relies on the action of central involutions in root $\SL_2$-subgroups: this will be clear in the proofs of Propositions \ref{p:sumofspots} and \ref{p:transitions}.
We do not know how to dispense with these involutions and it is not clear whether their role is that of mere accelerators or more essential.

The latter question makes sense for art's sake but there is in any case no hope to extend the theorem to characteristic $2$ (the Lie ring being of course left aside), since complete reducibility for quadratic $\SL_2(\K)$-modules could fail in characteristic $2$, a topic our utter ignorance prevents us from dwelling on.
\item
Since our linearisation argument goes uniformly and does not require any form of module identification, knowing the precise isomorphism type of $\bG$ is never necessary. A careful reader will wonder how much information is really needed, for instance what could be done without the assumption that $\bG$ is a simple algebraic group.
As a matter of fact, the very algebraicity of $\cG$ could perhaps be relaxed to something weaker.
There are indications that one does not need the torus to be $\K$-split, since most of the argument uses little semi-simple elements beyond involutions.
It might be enough to work in a more abstract setting than Chevalley groups: a group with some form of root datum, various root substructures being realised over various fields. In the end, the linear structure would arise only from the field associated to the root $\alpha_0$ of Notation \ref{n:alpha0}.
\end{itemize}

\section{The Proof}\label{S:proof}

For the reader's convenience let us state the theorem again.

\begin{theorem*}
Let $\K$ be a field of characteristic $\neq 2$ with more than three elements and $\bG$ be one of the simple algebraic groups (of classical or exceptional type; untwisted). Let $G = \bG_\K$ be the abstract group of $\K$-points of the functor $\bG$ and $\fg = (\Lie \bG)_\K$ be the abstract Lie \emph{ring} of $\K$-points of the functor $\Lie \bG$. Let $\cG$ be either $G$ or $\fg$ and $V$ be a $\Z[\cG]$-module.

Suppose that all roots act quadratically.
Then $V = \Triv_V(\cG) \oplus \lfloor \cG, V\rfloor$ and $\lfloor \cG, V\rfloor$ can be equipped with a $\K$-vector space structure making it isomorphic to a direct sum of minuscule representations of $\cG$ as a $\K[\cG]$-module.
\end{theorem*}

Recall from \cite[Chap. VIII, \S7.3, Proposition 6]{BLie78} that an irreducible representation of a semi-simple Lie algebra is minuscule iff for any weight $\mu$ and root $\alpha$, one has (with classical notations) $\mu(h_\alpha) \in \{-1, 0, 1\}$. For our purposes the latter property seems more tractable at first than a definition in terms of the action of the Weyl group.

The proof will therefore focus on weights and weight spaces. Of course in the absence of a field action the definition requires some care; the relevant analogues of weights and weight spaces will be called \emph{masses} and \emph{spots} (Definition \ref{d:spotsandmasses}).
We shall first show that the module is the direct sum of its spots (Proposition \ref{p:sumofspots}); the decomposition of an element of the module can be computed effectively. We then study how Weyl group elements permute spots. One point should be noted: if $\cG = G$, one can find elements lifting Weyl reflections in the normaliser of a maximal torus; if $\cG = \fg$, we must go to the enveloping ring since there are no suitable elements inside the Lie ring. But there is a slight trick enabling us to encode the Weyl group action in the latter case as well.
The action is then as expected (Proposition \ref{p:transitions}); the argument is the only step in the proof where we feel we actually do something.
Then Corollary \ref{c:Vcl} quickly enables us to reduce to an isotypical summand where the Weyl group acts transitively on masses. Once this is done we may define a field action on one arbitrary spot and use transitivity of the Weyl group to carry it around (Notation \ref{n:action}; here again the linear structure is given explicitly). Final linearity is easily proved in Proposition \ref{p:linear}.

\subsection{Prelude}\label{s:Prelude}

We presume the reader familiar with root systems, the Weyl group, and how elements of the normaliser of a maximal torus permute the various root subgroups: this will be one of the key ingredients of the proof. But we also wish to use products of ``root involutions'' in algebraic groups and the reader will find some refreshments here.

Fix a realisation of a simple algebraic group $G = \bG_\K$, and recall that central involutions of the various root $\SL_2$-subgroups can be computed from those attached to simple roots in the following way:
\begin{quote}
since for any root $\alpha$, one has $i_\alpha = \alpha^\vee(-1)$ (where $\alpha^\vee$ is the cocharacter $\K^\times \to T$, the torus), it suffices to express coroots in the cobasis.
\end{quote}
Consider for instance the case of $C_2$.

Normalising in such a way that long roots have Euclidean lengths $\sqrt{2}$, we may represent the dual system on the same picture, so that slightly abusing notations $\<\delta, \epsilon^\vee\>$ is given by $(\delta, \epsilon^\vee)$, where $\<\cdot, \cdot\>$ is the abstract root datum pairing and $(\cdot, \cdot)$ is the standard Euclidean dot product.
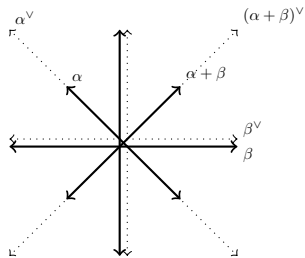
\[\begin{tikzpicture}[scale=.5]
     \draw[<->,thick] (-3,-.1) -- (3,-.1);
     \draw[<->,dotted] (-3,.1) -- (3,.1);
     \draw[<->,thick] (-.1,-3) -- (-.1,3);
     \draw[<->,dotted] (.1,-3) -- (.1,3);
    \foreach\ang in {45,135,...,315}{
     \draw[->,thick] (0,0) -- (\ang:2.12cm);
    }
\node[anchor=south west,scale=0.6] at (3,0) {$\beta^\vee$};
\node[anchor=south west,scale=0.6] at (-1.5,1.5) {$\alpha$};
    \foreach\ang in {45,135,...,315}{
     \draw[->,dotted] (0,0) -- (\ang:4.24cm);
    }
\node[anchor=north west,scale=0.6] at (3,0) {$\beta$};
\node[anchor=south west,scale=0.6] at (-3,3) {$\alpha^\vee$};
\node[anchor=south west,scale=0.6] at (1.5,1.5) {$\alpha+\beta$};
\node[anchor=south west,scale=0.6] at (3,3) {$(\alpha+\beta)^\vee$};
  \end{tikzpicture}
\]
It is clear from the latter picture that $(\alpha + \beta)^\vee = \alpha^\vee + 2 \beta^\vee$, and therefore $i_{\alpha + \beta} = i_\alpha i_\beta^2 = i_\alpha$.

The same picture allows of course to determine conjugates of root subgroups by elements of the Weyl group, with no computations. (We hope the following notations to be standard; in any case they will be introduced in Notation \ref{n:realisation:group} below.) Remember that $w_\gamma$ acts on root subgroups as $\sigma_\gamma$, the reflection with hyperplane $\gamma^\perp$, acts on roots: hence $w_\gamma U_\delta w_\gamma^{-1} = U_{\sigma_\gamma(\delta)}$ can be found graphically.

\subsection{Local Analysis}

The proof starts here. We may suppose the action to be non-trivial.
Let us first realise $\cG$, following the Chevalley-Steinberg ideology. We apologise for the necessarily heavy notations. As far as root data are concerned, these are fairly standard.

\begin{notation}[naming the root datum: $L, \Phi, L^\vee, \Phi^\vee, \<\cdot, \cdot\>, E, \sigma_\alpha, \Phi_+, \Phi_s$]
\label{n:rootdatum}\
\begin{itemize}
\item
Let $(L, \Phi, L^\vee, \Phi^\vee)$ be the root datum of $\bG$ and $\<\cdot, \cdot\> : L \times L^\vee \to \Z$ be the pairing; let $E = \R\otimes_\Z \Z\Phi$.
\item
For $\alpha \in \Phi$, let $\sigma_\alpha$ be the linear map on $E$ mapping $e$ to $e - \<e, \alpha^\vee\> \alpha$.
\item
Let $\Phi_+$ be a choice of positive roots and $\Phi_s$ be the (positive) simple roots.
\end{itemize}
\end{notation}

With this at hand we can realise $G$. General information on Chevalley groups can be found in \cite{CSimple} (in particular Chapters 5 and 6 there); sometimes our notations differ as we use $u$ for unipotent elements and $t$ for toral (semi-simple) elements; for elements associated to the Weyl group we use $w$. A particularly thorough reference is \cite[Expos\'e 23]{SGA3} but we shall avoid using geometric language.

\begin{notation}[realising $G$: $T, U_\alpha, G_\alpha, u_{\alpha, \lambda}, u_\alpha, w_\alpha, i_\alpha, t_{\alpha, \lambda}$]
\label{n:realisation:group}\
\begin{itemize}
\item
Fix an algebraic torus $T \leq G$; root subgroups will refer to this particular torus.
\item
For $\alpha \in \Phi$, let $U_\alpha$ be the root subgroup and $G_\alpha = G_{-\alpha} = \<U_\alpha, U_{-\alpha}\>$ be the root $\SL_2$-subgroup.
\item
Realising enables us to fix isomorphisms $\varphi_\alpha : \pSL_2(\K) \simeq G_\alpha$ mapping upper-triangular matrices to $U_\alpha$ and diagonal matrices to $T \cap G_\alpha$ and such that $\varphi_\alpha^{-1}\circ \varphi_{-\alpha} = \varphi_{- \alpha}^{-1}\circ \varphi_\alpha$ is the inverse-transpose automorphism.
Let:
\[u_{\alpha, \lambda} = \varphi_\alpha\left(\begin{pmatrix} 1 & \lambda \\ 0 & 1\end{pmatrix} \right), \quad t_{\alpha, \lambda} = \varphi_\alpha\left(\begin{pmatrix} \lambda \\  & \lambda^{-1}\end{pmatrix} \right)\]
For simplicity write $u_\alpha = u_{\alpha, 1}$.
\item
Let $w_\alpha = u_\alpha \cdot u_{-\alpha}^{-1}\cdot u_\alpha$ and $i_\alpha = w_\alpha^2$, an element with order at most $2$.

(It will be a consequence of Proposition \ref{p:localanalysis} below that $i_\alpha$ has order exactly $2$.)
\end{itemize}
\end{notation}

In particular, it should be noted that $w_{-\alpha} = w_\alpha$ and $w_\alpha u_{\alpha, \lambda} w_\alpha^{-1} = u_{-\alpha, \lambda}$. Moreover, $t_{-\alpha, \lambda} = t_{\alpha, \lambda^{-1}}$.
Importantly enough, $w_\alpha U_\beta w_\alpha^{-1} = U_{\sigma_\alpha(\beta)}$. Now to $\fg$.

\begin{notation}[realising $\fg$: $\ft, \fu_\alpha, \fg_\alpha, x_{\alpha, \lambda}, x_\alpha, h_{\alpha, \lambda}, h_\alpha$]
\label{n:realisation:ring}\
\begin{itemize}
\item
Fix a decomposition $\fg = \ft \oplus \oplus_{\alpha \in \Phi} \fu_\alpha$ with $\ft$ a Cartan subring and $\fu_\alpha$ the root subrings. Let $\fg_\alpha = \fg_{-\alpha} = \<\fu_\alpha, \fu_{-\alpha}\>$ be the root $\sl_2$-subring.
\item
Realising enables us to fix isomorphisms $\psi_\alpha : \sl_2(\K) \simeq \fg_\alpha$ mapping upper-triangular matrices to $\fg_\alpha$ and diagonal matrices to $\ft \cap \fg_\alpha$ and such that $\psi_\alpha^{-1}\circ \psi_{-\alpha} = \psi_{- \alpha}^{-1}\circ \psi_\alpha$ is the oppose-transpose automorphism.
Let:
\[x_{\alpha, \lambda} = \psi_\alpha\left(\begin{pmatrix} 0 & \lambda \\ 0 & 0\end{pmatrix} \right), \quad h_{\alpha, \lambda} = \psi_\alpha\left(\begin{pmatrix} \lambda \\  & - \lambda\end{pmatrix} \right)\]
For simplicity write $x_\alpha = u_{\alpha, 1}$ and $h_\alpha = h_{\alpha, 1}$.
\end{itemize}
\end{notation}


\begin{remark}\label{r:y}
If one were to let $y_{\alpha, \lambda} = \psi_\alpha\left(\begin{pmatrix} 0 & 0 \\ \lambda & 0\end{pmatrix} \right)$, one would have $y_{\alpha, \lambda} = - x_{-\alpha, \lambda} = x_{-\alpha, -\lambda}$. (The author finds computations less confusing to perform or check when working in the basis $(h, x, y)$.)
\end{remark}

Let us now provide uniform notations. The reason for choosing letter $\omega$ (which classically stands for the fundamental weights, see \cite{BLie78}) is by analogy with $w$ for elements of the group lifting the Weyl group. Checking that $\omega_\alpha$ behaves as expected in the case of the Lie ring too will be the object of Proposition \ref{p:localanalysis}.

\begin{notation}[realising $\cG$ by assembling Notations \ref{n:realisation:group} and \ref{n:realisation:ring}: $\cT, \cU_\alpha, \cG_\alpha, \omega_\alpha, \d_{\alpha, \lambda}, \tau_{\alpha, \lambda}$]
\label{n:realisation:uniform}\
\begin{itemize}
\item
If $\cG = G$ let $\cT = T$, let $\cU_\alpha = U_\alpha$ and $\cG_\alpha = G_\alpha$. Also let $\omega_\alpha = w_\alpha$; for $\lambda \in \K^\times$ let $\d_{\alpha, \lambda} = u_{\alpha, \lambda} - 1$ and $\tau_{\alpha, \lambda} = t_{\alpha, \lambda}$;
\item
if $\cG = \fg$ let $\cT = \ft$, let $\cU_\alpha = \fu_\alpha$ and $\cG_\alpha = \fg_\alpha$. Also let $\omega_\alpha = 1 - h_\alpha^2 + x_\alpha + x_{-\alpha}$; for $\lambda \in \K_+$ let $\d_{\alpha, \lambda} = x_{\alpha, \lambda}$ and $\tau_{\alpha, \lambda} = h_{\alpha, \lambda}$.
\end{itemize}
\end{notation}

Let us apologise again for this notation storm; on second thought, the reader will find that we mostly wanted to have toral and root elements, and to encode Weyl elements in a consistent way. The details of Notations \ref{n:realisation:group} and \ref{n:realisation:ring} may now be forgotten.
The first step is then just painting from nature.

\begin{proposition}[local analysis]\label{p:localanalysis}
For $\alpha \in \Phi$, one has $V = \Triv_V(\cG_\alpha) \oplus [\cG_\alpha, V]$ and $\lfloor\cG_\alpha, V\rfloor = \lfloor\cU_\alpha, V\rfloor \oplus \lfloor\cU_{-\alpha}, V\rfloor$ can be equipped with a $\K$-vector space structure making it isomorphic to a direct sum of natural representations of $\cG_\alpha$ as a $\K[\cG_\alpha]$-module.

Consequently:
\begin{itemize}
\item
in the case of the group, $C_V(G_\alpha) = C_V(i_\alpha)$ and $[G_\alpha, V] = [i_\alpha, V]$;
\item
$\omega_\alpha = \omega_{-\alpha}$ is a bijection fixing $\Triv_V(\cG_\alpha)$ pointwise and mapping $\lfloor\cU_\alpha, V\rfloor$ to $\lfloor\cU_{-\alpha}, V\rfloor$ and conversely; $\omega_\alpha^2$ acts as $-1$ on $\lfloor\cG_\alpha, V\rfloor$;
\item
for $v \in \lfloor\cU_\alpha, V\rfloor = \Triv_{\lfloor \cG_\alpha, V\rfloor}(\cU_\alpha)$, one has $\d_{\alpha, \lambda} \omega_\alpha v = - \tau_{\alpha, \lambda} v$;
\item
for $\alpha \in \Phi$, one has $\omega_\alpha \d_{\alpha, \lambda} \omega_\alpha^{-1} = \d_{-\alpha, \lambda}$; moreover $\omega_\alpha$ normalises the image of $\cT$ in $\End(V)$.
\end{itemize}
\end{proposition}
\begin{proof}
By assumption $\cG_\alpha$ acts quadratically. So most claims follow from Timmesfeld's Quadratic Theorem and its Lie-ring analogue from the introduction, and inspection in the natural $\bSL_2$-module (possibly with a few computations).

We urge the reader not to overlook the fact that $\d_{\alpha, \lambda} \omega_\alpha v = -\tau_{\alpha, \lambda} v$ for $v \in \lfloor \cU_\alpha, V\rfloor$: here we can see it by inspection again, but this rather deep equation is the crux of the Quadratic Theorem. (The curious reader interested in extending this remarkable formula to other rational representations may be directed to the proof of \cite[Theorem 2]{TV-III}.)

We now prove the final statement, and begin with $\omega_\alpha \d_{\alpha, \lambda} \omega_\alpha^{-1} = \d_{-\alpha, \lambda}$.
This is clear in the case of the group (the formula holds in $G_\alpha$, replacing $\d$ by $u$); for the Lie ring, proceed piecewise. On $\Ann_V(\fg_\alpha)$ this is clear as both hands are zero. So let us work on $\fg_\alpha \cdot V$, where $h_\alpha^2 = 1$, so that $\omega_\alpha$ simplifies into $x_\alpha + x_{-\alpha}$. Let $\llbracket f, g\rrbracket = fg-gf$ in $\End(V)$ (we avoid $[\cdot, \cdot]$ which we reserve for group commutators).
Then using quadraticity of $\fg_\alpha$ and with a possible look at Remark \ref{r:y} one can check:
\begin{align*}
\omega_\alpha \d_{\alpha, \lambda} \omega_{\alpha}^{-1} &
= - (x_\alpha + x_{-\alpha}) x_{\alpha, \lambda} (x_\alpha + x_{-\alpha}) = - x_{-\alpha} x_{\alpha, \lambda}  x_{-\alpha}\\
&  = - \llbracket x_{-\alpha}, x_{\alpha, \lambda}\rrbracket x_{-\alpha}
= - h_{\alpha, \lambda} x_{-\alpha}
 = - \llbracket x_{-\alpha, \lambda}, x_\alpha\rrbracket x_{-\alpha}\\
& = - x_{-\alpha, \lambda} \llbracket x_\alpha, x_{-\alpha}\rrbracket
= x_{-\alpha, \lambda} h_\alpha
 = x_{-\alpha, \lambda}
 = \d_{-\alpha, \lambda}
\end{align*}

We now show that $\omega_\alpha$ normalises the image in $\End(V)$ of $\cT$: here again the case of the group is obvious so we turn to the Lie ring.
Let $\alpha, \beta$ be roots. Then:
\begin{itemize}
\item
first, $\llbracket \omega_\alpha, \tau_{\beta, \lambda}\rrbracket = \llbracket x_\alpha + x_{-\alpha}, h_{\beta, \lambda}\rrbracket = x_{ - \alpha, \<\alpha, \beta^\vee\> \lambda} - x_{\alpha, \<\alpha, \beta^\vee\> \lambda}$;
\item
by piecewise inspection one has $x_{-\alpha, \lambda} x_\alpha - x_{\alpha, \lambda} x_{-\alpha} = h_{\alpha, \lambda}$ and $x_{\alpha, \lambda} h_\alpha = - x_{\alpha, \lambda}$ in $\End(V)$; also notice that $\omega_\alpha^{-1} = - x_\alpha - x_{-\alpha} + 1 - h_\alpha^2$;
\item
therefore $(x_{- \alpha, \lambda} - x_{\alpha, \lambda}) \omega_\alpha^{-1} = (x_{ -\alpha, \lambda} - x_{\alpha, \lambda}) (- x_\alpha - x_{-\alpha}) = - x_{-\alpha, \lambda} x_\alpha + x_{\alpha, \lambda} x_{-\alpha} = - h_{\alpha, \lambda}$;
\item
as a consequence $\omega_\alpha \tau_{\beta, \lambda} \omega_\alpha^{-1} = \left(\llbracket \omega_\alpha, \tau_{\beta, \lambda} \rrbracket + \tau_{\beta, \lambda} \omega_\alpha\right) \omega_\alpha^{-1} = (x_{- \alpha, \<\alpha, \beta^\vee\> \lambda} - x_{\alpha, \<\alpha, \beta^\vee\> \lambda}) \omega_\alpha^{-1} + h_{\beta, \lambda} = - h_{\alpha, \<\alpha, \beta^\vee\> \lambda} + h_{\beta, \lambda}$.
\end{itemize}
Hence $\omega_\alpha$ normalises the image of $\ft$.
\end{proof}

Notice that if $\cG = G$, then since $\bG$ is simple as an algebraic group and the action is non-trivial, for any $\alpha \in \Phi$ the root $\bSL_2$-subgroup $G_\alpha$ must act non-trivially on $V$. As a consequence $G_\alpha \simeq \SL_2(\K)$ and $i_\alpha$ is a genuine involution. (Without simplicity, one would just factor out and use the root datum of the quotient.)

\subsection{Spots and Masses}

Capturing weight spaces requires a little care in the absence of a linear structure.

In the case of the Lie ring $\cG = \fg$ there is a straightforward approach. Diagonalise all operators $h_\alpha$ ($\alpha \in \Phi_s$) \emph{simultaneously}; the various eigenspaces will be the weight spaces for the action of $\ft$. But there is no such argument in the case of the group $\cG = G$. Yet returning to the Lie ring one sees by inspection that $\ker(h_\alpha - 1) = \fu_\alpha \cdot V = \lfloor \cU_\alpha, V\rfloor$ whereas $\ker (h_\alpha) = \Ann_V(\fg_\alpha) = \Triv_V(\cG_\alpha)$.
This suggests a general method.

Recall from Notation \ref{n:rootdatum} that $\Phi$ (resp. $\Phi^\vee$) denotes the root (resp. dual root) system and $E$ the underlying Euclidean space. We had also let $\<\cdot, \cdot\> : L \times L^\vee \to \Z$ denote the pairing.

\begin{definition}\label{d:spotsandmasses}\
\begin{itemize}
\item
For $\mu \in E$ and $\alpha$ a root define $V_{(\mu, \alpha^\vee)}$ as follows:
\begin{itemize}
\item
if $\<\mu, \alpha^\vee\> = -1$ let $V_{(\mu, \alpha^\vee)} = \lfloor \cU_{-\alpha}, V\rfloor$;
\item
if $\<\mu, \alpha^\vee\> = 0$ let $V_{(\mu, \alpha^\vee)} = \Triv_V(\cG_\alpha)$;
\item
if $\<\mu, \alpha^\vee\> = 1$ let $V_{(\mu, \alpha^\vee)} = \lfloor \cU_\alpha, V\rfloor$;
\item
if $\<\mu, \alpha^\vee\> \notin\{-1, 0, 1\}$ let $V_{(\mu, \alpha^\vee)} = \{0\}$.
\end{itemize}
\item
For $\mu \in E$ let $S_\mu = \bigcap_{\alpha \in \Phi_s} V_{(\mu, \alpha^\vee)}$ be the \emph{spot with mass $\mu$}.
\item
Let $M = \{\mu \in E: S_\mu \neq \{0\}\}$ be the set of \emph{masses}. (Being a mass certainly implies: $\forall \alpha \in \Phi_s, \<\mu, \alpha^\vee\> \in \{-1, 0, 1\}$, but the condition is not sufficient.)
\end{itemize}
\end{definition}

Notice how in the presence of a suitable field action, a mass $\mu$ will become a weight.
The idea in taking the intersection over the set of \emph{simple} roots $\Phi_s$ is that the behaviour of simple roots determines that of all roots; we shall not need nor prove this.

\begin{remark}\label{r:spots}\
\begin{enumerate}
\item
$S_\mu$ is a $\cT$-submodule of $V$.
\item
Suppose $\cG = \fg$ and $\<\mu, \alpha^\vee\> \in \{-1, 0, 1\}$. Then $V_{(\mu, \alpha^\vee)} = \ker(h_\alpha - \<\mu, \alpha^\vee\>)$.
\end{enumerate}
\end{remark}

We now show that $V$ is the direct sum of its spots.

\begin{proposition}\label{p:sumofspots}
$V = \oplus_{\mu \in M} S_\mu$.
\end{proposition}

\begin{remark}\label{r:G2trivial}
As a matter of fact during the proof we shall see that if $\cG = G$ and $\bG = G_2$, then $V = S_0$ (with $0$ the null mass).
\end{remark}

\begin{proof}
In the case of the Lie ring $\fg$ this is obvious since by Proposition \ref{p:localanalysis}, the operators $h_\alpha$ are simultaneously diagonalisable, with eigenvalues $\{-1, 0, 1\}$.  We then focus on the case of the group $G$; nothing so quick is available, since no toral element in $G_\alpha$ suffices to determine the value of $\<\mu, \alpha^\vee\>$: looking at the involution can distinguish $0$ from $\pm 1$, but no further. We need a closer look.

Bear in mind from Proposition \ref{p:localanalysis} that for any $\alpha \in \Phi$, one has $C_V(i_\alpha) = C_V(G_\alpha)$ and $[i_\alpha, V] = [G_\alpha, V]$; also 
$[U_\alpha, V] = C_{[G_\alpha, V]}(U_\alpha)$. Finally if $v \in [U_\alpha, V]$, then $\d_\alpha w_\alpha v = - v$.

\begin{claim}
The sum is direct.
\end{claim}
\begin{proofclaim}
Let $\sum_{\mu \in M_1} v_\mu = 0$ be an identity minimal with respect to: for all $\mu \in M_1$, $v_\mu \in S_\mu\setminus\{0\}$. Let $\nu \in M_1$ (if any) and $\alpha \in \Phi_s$ be fixed.
\begin{itemize}
\item
If $\<\nu, \alpha^\vee\> = 0$ then $i_\alpha v_\nu = v_\nu$, so that $\sum_{\mu \in M_1} (i_\alpha v_\mu - v_\mu) = 0$ is a shorter relation. It follows that $i_\alpha v_\mu = v_\mu$ for all $\mu \in M_1$, meaning $\<\mu, \alpha^\vee\> = 0$.
\item
If $\<\nu, \alpha^\vee\> = 1$ then $\d_\alpha w_\alpha v_\nu = -v_\nu$; notice that whenever $\<\mu, \alpha^\vee\> \neq 1$, one has $\d_\alpha w_\alpha v_\mu = 0$. So minimality again forces $\<\mu, \alpha^\vee\> = 1$ for all $\mu \in M_1$.
\item
There is a similar argument if $\<\nu, \alpha^\vee\> = -1$.
\end{itemize}
This shows that all $\mu \in M_1$ coincide on all $\alpha \in \Phi_s$, a spanning set of $E$: $M_1$ is at most a singleton, hence empty, as desired.
\end{proofclaim}

Now let $R(V) = \oplus_{\mu \in M} S_\mu$ and fix $v \in V$; we aim at showing $v \in R(V)$.

Let $\alpha \in \Phi_s$ be extremal in the Dynkin diagram and $\beta$ be its neighbour; \emph{we may suppose $\alpha$ not to be longer than $\beta$}. By induction, we may assume that for any $\gamma \in \Phi_s\setminus\{\alpha\}$, the element $v$ is \emph{decomposed} under the action of $G_\gamma$, viz.:
\[v \in C_V(i_\gamma) \cup [U_\gamma, V] \cup [U_{-\gamma}, V]\]

\begin{claim}
We may suppose $v \in [i_\alpha, V]$.
\end{claim}
\begin{proofclaim}
Write $v = v_0 + v_\pm$ with respect to the action of $i_\alpha$, meaning $v_0 \in C_V(i_\alpha) = C_V(G_\alpha)$ and $v_\pm \in [i_\alpha, V] = [G_\alpha, V]$; as a matter of fact $v_\pm = \frac{1}{2} [i_\alpha, v]$ (which makes sense since $[i_\alpha, V]$ is a vector space over $\K$).
Since $i_\alpha$ centralises $G_\gamma$ for $\gamma \in \Phi_s\setminus\{\alpha, \beta\}$, $v_0$ remains decomposed under the action of such root $\SL_2$-subgroups; by construction, it is decomposed under that of $G_\alpha$.
Now $i_\alpha$ normalises $U_\beta$ and $U_{-\beta}$ (hence also $G_\beta$). As a consequence:
\begin{itemize}
\item
if $v \in C_V(G_\beta)$ then $i_\alpha v$, $v_\pm$, and $v_0$ lie in $C_V(G_\beta)$;
\item
if $v \in [U_\beta, V]$ then $i_\alpha v$, $v_\pm$, and $v_0$ lie in $[U_\beta, V]$;
\item
there is a similar argument if $v \in [U_{-\beta}, V]$.
\end{itemize}
%
%
%
As a conclusion, $v_0$ is decomposed under the action of $G_\beta$ as well: hence $v_0 \in R(V)$. We may therefore assume $v = v_\pm \in [i_\alpha, V]$.
\end{proofclaim}

It follows from inspection in $\Nat \SL_2(\K)$ that $v = v_+ + v_-$ with $v_+ = - \d_\alpha w_\alpha v \in [U_\alpha, V]$ and $v_- = - w_\alpha \d_\alpha v \in [U_{-\alpha}, V]$.
We aim at showing $v_+, v_- \in R(V)$. By construction the latter elements are already decomposed under the action of $G_\alpha$ and $G_\gamma$ for $\gamma \in \Phi_s\setminus\{\alpha, \beta\}$, but it remains to see what happened under the action of $G_\beta$.
This we do dividing three cases (remember that we assumed $\alpha$ not to be longer than $\beta$). We use classical notations for Dynkin diagrams: $\simple$, $=$, and $\equiv$; an arrow goes from a long root to a short root.

\begin{claim}
If $\alpha \simple \beta$, then $v \in R(V)$.
\end{claim}
\begin{proofclaim}
\[\begin{tikzpicture}[scale=.5]\label{f:A2}
    \foreach\ang in {60,120,...,360}{
     \draw[->,thick] (0,0) -- (\ang:3cm);
    }
\node[anchor=south,scale=0.6] at (3,0) {$\beta$};
\node[anchor=east,scale=0.6] at (-1.5,2.6) {$\alpha$};
\node[anchor=west,scale=0.6] at (1.5,2.6) {$\alpha + \beta$};
  \end{tikzpicture}
\]

Here 
$i_{\alpha + \beta} = i_\alpha i_\beta$ (the reader may wish to return to \S\ref{s:Prelude}); also notice that $i_\beta w_\alpha = w_\alpha i_\alpha i_\beta$. Since $\<\alpha, \beta^\vee\> = - 1$, the involution $i_\beta$ inverts $U_\alpha$, observe that $i_\beta \d_\alpha w_\alpha v = - \d_\alpha i_\beta w_\alpha v = - \d_\alpha w_\alpha i_\alpha i_\beta v$.

\begin{itemize}
\item
If $v \in C_V(i_\beta)$, then $\d_\alpha w_\alpha v \in C_V(i_\beta) = C_V(G_\beta)$; hence $v_+, v_- \in R(V)$: we are done.
\item
If $v \in [U_\beta, V]$, then both $i_\beta$ and $i_\alpha$ invert $v$: hence $i_{\alpha + \beta} = i_\alpha i_\beta$ centralises $v$, so that $v = w_{\alpha + \beta} v \in [w_{\alpha + \beta} U_\beta w_{\alpha + \beta}^{-1}, V] = [U_{-\alpha}, V]$. Hence $v \in R(V)$.
\item
Likewise, if $v \in [U_{-\beta}, V]$, then $v \in [U_\alpha, V]$.\qedhere
\end{itemize}
\end{proofclaim}

\begin{claim}
If $\alpha \doubleleq \beta$, then $v \in R(V)$.
\end{claim}
\begin{proofclaim}
\[\begin{tikzpicture}[scale=.5]\label{f:B2}
    \foreach\ang in {90,180,...,360}{
     \draw[->,thick] (0,0) -- (\ang:3cm);
    }
    \foreach\ang in {45,135,...,315}{
     \draw[->,thick] (0,0) -- (\ang:2.12cm);
     \draw[->,dotted] (0,0) -- (\ang:2*2.12cm);
    }
\node[anchor=west,scale=0.6] at (3,0) {$\beta$};
\node[anchor=east,scale=0.6] at (-1.5,1.5) {$\alpha$};
\node[anchor=west,scale=0.6] at (1.5,1.5) {$\alpha + \beta$};
\node[anchor=south,scale=0.6] at (0,3) {$2 \alpha + \beta$};
  \end{tikzpicture}
\]

Now $i_{2\alpha + \beta} = i_\alpha i_\beta$ and $i_\beta w_\alpha = w_\alpha i_{2\alpha + \beta}$.
Since $\<\alpha, \beta^\vee\> = - 1$, the involution $i_\beta$ inverts $U_\alpha$, and one still has $i_\beta \d_\alpha w_\alpha v = - \d_\alpha w_\alpha i_\alpha i_\beta v$.

\begin{itemize}
\item
If $v \in C_V(i_\beta)$, then $i_\beta \d_\alpha w_\alpha v = \d_\alpha w_\alpha v$, so $v_+$ lies in $C_V(i_\beta)$; since $v$ as well, so does $v_-$. As a consequence $v_+, v_- \in R(V)$.
\item
If $v \in [U_\beta, V]$, then $w_\alpha v \in [w_\alpha U_\beta w_\alpha^{-1}, V] = [U_{2\alpha + \beta}, V]$. Now $[U_\alpha, U_{2\alpha + \beta}] = 1$ in the group, so by the three subgroups lemma $v_+ = - \d_\alpha w_\alpha v \in [U_{2\alpha + \beta}, V] \leq [i_{2\alpha + \beta}, V]$.

However $i_\alpha i_\beta \d_\alpha w_\alpha v = \d_\alpha w_\alpha v$ so $v_+ \in C_V(i_\alpha i_\beta) = C_V(i_{2\alpha + \beta})$.
This shows $v_+ = 0$, and therefore $v = v_- \in R(V)$.
%
%
\item
There is a similar argument showing $v = v_+ \in R(V)$ if $v \in [U_{-\beta}, V]$.
\qedhere
\end{itemize}
\end{proofclaim}

\begin{claim}
If $\alpha \tripleleq \beta$, then $v = 0 \in R(V)$.
\end{claim}
\begin{proofclaim}
\[\begin{tikzpicture}[scale=.5]
    \foreach\ang in {60,120,...,360}{
     \draw[->,thick] (0,0) -- (\ang:3cm);
    }
    \foreach\ang in {30,90,...,330}{
     \draw[->,thick] (0,0) -- (\ang:1.73cm);
\draw[->,dotted] (0,0) -- (\ang:3*1.73cm);
    }
\node[anchor=west,scale=0.6] at (3,0) {$\beta$};
\node[anchor=east,scale=0.6] at (-1.73,1) {$\alpha$};
\node[anchor=west,scale=0.6] at (1.73,1) {$\alpha + \beta$};
\node[anchor=south,scale=0.6] at (0,2) {$2\alpha + \beta$};
\node[anchor=south,scale=0.6] at (1.5,2.6) {$3 \alpha + 2\beta$};
\node[anchor=south,scale=0.6] at (- 1.5,2.6) {$3 \alpha + \beta$};
  \end{tikzpicture}
\]

Finally $i_{\alpha + \beta} = i_\alpha i_\beta = i_{3\alpha + \beta}$; also $i_{2\alpha + \beta} = i_\beta$ and $i_{3\alpha + 2\beta} = i_\alpha$.

\begin{itemize}
\item
If $v \in C_V(i_\beta)$ then it can be checked that $i_\beta \d_\alpha w_\alpha v = - \d_\alpha w_\alpha i_\alpha i_\beta v = \d_\alpha w_\alpha v$, implying that $v_+ = w_\beta v_+ \in [w_\beta U_\alpha w_\beta^{-1}, V] = [U_{\alpha + \beta}, V]$ and $v_+ = w_{2\alpha + \beta} v_+ \in [w_{2\alpha + \beta} U_\alpha w_{2\alpha + \beta}^{-1}, V] = [U_{-\alpha - \beta}, V]$, so $v_+ = 0$. One can show $v_- = 0$ as well; hence $v = 0 \in R(V)$.
\item
If $v \in [U_\beta, V]$, then both $i_\alpha$ and $i_\beta$ invert $v$; as a consequence one has $v = w_{\alpha + \beta} v \in [w_{\alpha + \beta} U_\beta w_{\alpha + \beta}^{-1}, V] = [U_{-3\alpha - 2\beta}, V]$ and $v = w_{3\alpha + \beta} v \in [w_{3\alpha + \beta} U_\beta w_{3\alpha + \beta}^{-1}, V] = [U_{3\alpha + 2\beta}, V]$, so $v = 0 \in R(V)$.
\item
There is a similar argument if $v \in [U_{-\beta}, V]$.
\end{itemize}
Notice that in case $\bG = G_2$ we proved $v = v_0$ with the above notations. This means that $V$ is centralised by $i_\alpha$ and therefore by $G_\alpha$, so by simplicity of $\bG$ the action of $G$ on $V$ is actually trivial (or at least, without simplicity, there is something to factor out).
\end{proofclaim}

This completes the proof of Proposition \ref{p:sumofspots}.
\end{proof}

\subsection{Weyl Group Action}

We now wish to see how the Weyl group permutes spots: it is as expected, with the major warning that it is not entirely clear what this means in the case of the Lie ring (see \S\ref{s:remarks} for a warning, and remember our contortions in Notation \ref{n:realisation:uniform}). Our approach is elementary again.

In Notation \ref{n:rootdatum}, for any $\alpha \in \Phi$ we introduced the reflection $\sigma_\alpha(e) = e - \<e, \alpha^\vee\> \alpha$.
Also remember from Notation \ref{n:realisation:uniform} that we have let $\omega_\alpha = w_\alpha$ if $\cG = G$ and $\omega_\alpha = 1 - h_\alpha^2 + x_\alpha - x_{-\alpha}$ if $\cG = \fg$; the action of $\omega_\alpha$ is as expected by Proposition \ref{p:localanalysis}.

Before the statement, observe that will shall be woriking with simple roots throughout. The author did not think about extending to other roots; this will not be necessary.

\begin{proposition}\label{p:transitions}
For all $(\alpha, \mu) \in \Phi_s \times M$, one has $\omega_\alpha S_\mu = S_{\sigma_\alpha (\mu)}$.
\end{proposition}


\begin{proof}
The case of the Lie ring is straightforward and will be dealt with quickly.

\begin{claim}
We may suppose $\cG = G$.
\end{claim}
\begin{proofclaim}
Suppose $\cG = \fg$; let $\mu \in M$ be a mass; let $\alpha, \beta$ be any two (possibly equal) simple roots.
First notice that in the Lie ring $(\End(V), +, \llbracket \cdot, \cdot \rrbracket)$, one has $\llbracket h_\beta, \omega_\alpha\rrbracket = \llbracket h_\beta, x_\alpha + x_{-\alpha}\rrbracket = \<\alpha, \beta^\vee\> (x_\alpha - x_{-\alpha})$. On the other hand, as one checks by piecewise inspection with the help of Proposition \ref{p:localanalysis}, for $v \in V_{(\mu, \alpha^\vee)}$ holds: $(x_{-\alpha} - x_\alpha) v = - \<\mu, \alpha^\vee\> \omega_\alpha v$.
So for $v \in S_\mu \leq V_{(\mu, \alpha^\vee)} \cap V_{(\mu, \beta^\vee)}$, one has:
\begin{align*}
(h_\beta - \<\sigma_\alpha(\mu), \beta^\vee\>)\omega_\alpha v & = \left(\omega_\alpha h_\beta + \<\alpha, \beta^\vee\> (x_\alpha - x_{-\alpha}) - (\<\mu, \beta^\vee\> - \<\mu, \alpha^\vee\>\<\alpha, \beta^\vee\>) \omega_\alpha\right) v\\
& = \omega_\alpha \left(\<\mu, \beta^\vee\> - \<\alpha, \beta^\vee\>\<\mu, \alpha^\vee\> - \<\mu, \beta^\vee\> + \<\mu, \alpha^\vee\>\<\alpha, \beta^\vee\>\right) v\\
& = 0
\end{align*}
showing that $\omega_\alpha S_\mu \leq \ker (h_\beta - \<\sigma_\alpha(\mu), \beta^\vee\>)$.

We claim that $\ker (h_\beta - \<\sigma_\alpha(\mu), \beta^\vee\>) = V_{(\sigma_\alpha(\mu), \beta^\vee)}$; by construction (see Remark \ref{r:spots}) it suffices to see why $\<\sigma_\alpha(\mu), \beta^\vee\> \in \{-1, 0, 1\}$.  But let $\gamma \in E$ satisfy $\gamma^\vee = \sigma_{\alpha}^\vee (\beta^\vee) = \beta^\vee - \<\alpha, \beta^\vee\> \alpha^\vee$; we know that $\gamma \in \Phi$ (not necessarily simple though). Now,
\[h_\gamma = \gamma^\vee(1) = \beta^\vee (1) - \<\alpha, \beta^\vee\> \alpha^\vee (1) = h_\beta - \<\alpha, \beta^\vee\> h_\alpha\]
acts on $S_\mu$ as the integer
\[\<\mu, \beta^\vee\> - \<\alpha, \beta^\vee\> \<\mu, \alpha^\vee\> = \<\sigma_\alpha(\mu), \beta^\vee\>\]
Since $\fg_\gamma$ is quadratic -- bear in mind the assumption was on \emph{all} roots -- this integer remains in $\{-1, 0, 1\}$, as desired.

Therefore $\omega_\alpha S_\mu \leq \ker (h_\beta - \<\sigma_\alpha(\mu), \beta^\vee\>) = V_{(\sigma_\alpha(\mu), \beta^\vee)}$.
Since this holds for any $\beta \in \Phi_s$, one has $\omega_\alpha S_\mu \leq S_{\sigma_\alpha(\mu)}$. Since this holds for any mass $\mu \in M$, one also finds $\omega_\alpha S_{\sigma_\alpha(\mu)} \leq S_\mu$, proving equality.
\end{proofclaim}

We move to the case of the group, for which there is no such argument: exactly like in Proposition \ref{p:sumofspots}, no toral element in $G_\alpha$ can capture $\<\mu, \alpha^\vee\>$.

\begin{claim}
We may assume $\<\mu, \alpha^\vee\> = 1$; it suffices to prove that for any $\beta \in \Phi_s$,
\[\omega_\alpha S_\mu \leq V_{(\mu - \alpha, \beta^\vee)} \quad(*)\]
\end{claim}
\begin{proofclaim}
First suppose $\<\mu, \alpha^\vee\> = 0$. Then $\sigma_\alpha(\mu) = \mu$ and $\omega_\alpha$ acts as $\Id$ on $S_\mu$: there is nothing to prove. We then turn to $\<\mu, \alpha^\vee\> = \pm 1$. Observe how it suffices to check $\omega_\alpha S_\mu \leq S_{\sigma_\alpha(\mu)}$:  then one will find $S_\mu = \omega_\alpha^2 S_\mu \leq \omega_\alpha S_{\sigma_\alpha(\mu)} \leq S_\mu$, proving equality.

So it suffices to see that $\omega_\alpha S_\mu \leq S_{\sigma_\alpha(\mu)}$; by symmetry, we may assume $\<\mu, \alpha^\vee\> = 1$, so that $\sigma_\alpha(\mu) = \mu - \alpha$. We then wish to show $\omega_\alpha S_\mu \leq S_{\mu - \alpha}$. This we shall do by taking another simple root $\beta \in \Phi_s$ and showing that the action of $G_\beta$ on $\omega_\alpha S_\mu$ is as expected, viz.
condition $(*)$ above.
\end{proofclaim}

We start a case divison based on the nature of the bound between $\alpha$ and $\beta$.

\begin{claim}
If $\beta$ is not bound to $\alpha$ then $(*)$ holds.
\end{claim}
\begin{proofclaim}
%
%
%

If $\beta$ equals $\alpha$ then with the assumption that $\<\mu, \alpha^\vee\> = 1$, one finds $S_\mu \leq V_{(\mu, \alpha^\vee)} = \lfloor \cU_\alpha, V\rfloor$, and:
\[\omega_\alpha S_\mu \leq \omega_\alpha \lfloor \cU_\alpha, V\rfloor = \lfloor \cU_{-\alpha}, V\rfloor = V_{(\sigma_\alpha(\mu), \alpha^\vee)}\qedhere\]
If $\beta$ is neither bound nor equal to $\alpha$, then $(*)$ is obvious since the images of $\cG_\alpha$ and $\cG_\beta$ in $\End(V)$ commute, and $\<\sigma_\alpha(\mu), \beta^\vee\> = \<\mu, \beta^\vee\>$.
\end{proofclaim}


\begin{claim}
If $\alpha \simple \beta$, then $(*)$ holds.
\end{claim}
\begin{proofclaim}
There is a picture on page \pageref{f:A2}; in particular bear in mind that $i_{\alpha + \beta} = i_\alpha i_\beta$.
Also notice that $\<\mu - \alpha, \beta^\vee\> = \<\mu, \beta^\vee\> + 1$.
\begin{itemize}
\item
Suppose $\<\mu, \beta^\vee\> = -1$; notice that $\<\mu - \alpha, \beta^\vee\> = 0$. Since both $i_\alpha$ and $i_\beta$ invert $S_\mu$, one has $S_\mu \leq C_V(i_{\alpha + \beta}) = C_V(G_{\alpha + \beta})$, and $w_\alpha S_\mu \leq C_V(w_\alpha G_{\alpha + \beta} w_\alpha^{-1}) = C_V(G_\beta) = V_{(\mu - \alpha, \beta^\vee)}$.
\item
Now suppose $\<\mu, \beta^\vee\> = 0$; hence $\<\mu - \alpha, \beta^\vee\> = 1$. 
Then $i_\beta$ centralises $S_\mu$, so $S_\mu = w_\beta S_\mu \leq [w_\beta U_\alpha w_\beta^{-1}, V] = [U_{\alpha + \beta}, V]$. Hence $w_\alpha S_\mu \leq [w_\alpha U_{\alpha + \beta}w_\alpha^{-1}, V] = [U_\beta, V] = V_{(\mu - \alpha, \beta^\vee)}$.
\item
Finally suppose $\<\mu, \beta^\vee\> = 1$; notice that now $\<\mu - \alpha, \beta^\vee\> = 2$ and there is a contradiction in the air.
Here again, both $i_\alpha$ and $i_\beta$ invert $S_\mu$ so $i_{\alpha + \beta}$ centralises it. Therefore $S_\mu = w_{\alpha + \beta} S_\mu \leq [w_{\alpha + \beta} U_\alpha w_{\alpha+ \beta}^{-1}, V] = [U_{-\beta}, V]$, and $S_\mu \leq [U_\beta, V] \cap [U_{-\beta}, V]= 0$. This is a contradiction to $\mu \in M$, that is, $S_\mu \neq 0$ (see Definition \ref{d:spotsandmasses}).
\qedhere
\end{itemize}
\end{proofclaim}


\begin{claim}
If $\alpha \doubleleq \beta$ then $(*)$ holds.
\end{claim}
\begin{proofclaim}
There is a picture on page \pageref{f:B2}; one has $i_{\alpha + \beta} = i_\alpha$ and $i_{2 \alpha + \beta} = i_\alpha i_\beta$. Notice that $\<\mu - \alpha, \beta^\vee\> = \<\mu, \beta^\vee\> + 1$.
\begin{itemize}
\item
Suppose $\<\mu, \beta^\vee\> = -1$, so that $\<\mu - \alpha, \beta^\vee\> = 0$. Both $i_\alpha$ and $i_\beta$ invert $S_\mu$, so $S_\mu \leq C_V(i_\alpha i_\beta) = C_V(i_{2\alpha + \beta})$ and $w_\alpha S_\mu \leq C_V(w_\alpha G_{2\alpha + \beta} w_\alpha^{-1}) = C_V(G_\beta)$, as desired.
\item
Now suppose $\<\mu, \beta^\vee\> = 0$, so that $\<\mu - \alpha, \beta^\vee\> = 1$. 
Then $U_\alpha$, $U_\beta$, and therefore $U_{2\alpha + \beta}$ as well, centralise $S_\mu$. On the other hand $i_{2\alpha + \beta} = i_\alpha i_\beta$ inverts it, so $S_\mu \leq [i_{2\alpha + \beta}, C_V(U_{2\alpha + \beta})] = [U_{2\alpha + \beta}, V]$ and $w_\alpha S_\mu \leq [w_\alpha U_{2\alpha + \beta} w_\alpha^{-1}, V] = [U_\beta, V] = V_{(\mu - \alpha, \beta^\vee)}$.
\item
Finally suppose $\<\mu, \beta^\vee\> = 1$, so that $\<\mu - \alpha, \beta^\vee\> = 2$. Here again both $i_\alpha$ and $i_\beta$ invert $S_\mu$: so $i_{\alpha + 2\beta}$ centralises it, and therefore $S_\mu = w_{\alpha + 2\beta} S_\mu \leq [w_{\alpha + 2\beta} U_\alpha w_{\alpha + 2\beta}^{-1}, V] = [U_{-\alpha - \beta}, V]$. But $U_\alpha$, $U_\beta$, and therefore $U_{\alpha + \beta}$ as well, centralise $S_\mu$, showing $S_\mu = 0$: against $\mu \in M$.
\qedhere
\end{itemize}
\end{proofclaim}


\begin{claim}
If $\alpha \doublegeq \beta$ then $(*)$ holds.
\end{claim}
\begin{proofclaim}
Be careful that $\beta$ is now the short root; hence $i_{\alpha + \beta} = i_\beta$ and $i_{\alpha + 2\beta} = i_\alpha i_\beta$. Notice that $\<\mu - \alpha, \beta^\vee\> = \<\mu, \beta^\vee\> + 2$.
\begin{itemize}
\item
Suppose $\<\mu, \beta^\vee\> = -1$, so that $\<\mu - \alpha, \beta^\vee\> = 1$. Then both $i_\alpha$ and $i_\beta$ invert $S_\mu$: so $i_{\alpha + 2\beta}$ centralises it, and therefore $S_\mu = w_{\alpha + 2\beta} S_\mu \leq  [w_{\alpha + 2\beta} U_{- \beta} w_{\alpha + 2\beta}^{-1}, V] = [U_{\alpha + \beta}, V]$. Hence $w_\alpha S_\mu \leq [w_\alpha U_{\alpha + \beta} w_\alpha^{-1}, V] = [U_\beta, V]$.
\item
Now suppose $\<\mu, \beta^\vee\> = 0$, so that $\<\mu - \alpha, \beta^\vee\> = 2$. Then $i_{\alpha + \beta} = i_\beta$ centralises $S_\mu$: as a consequence $S_\mu = w_{\alpha + \beta} S_\mu \leq [w_{\alpha + \beta} U_\alpha w_{\alpha + \beta}^{-1}, V] = [U_{-\alpha - 2\beta}, V]$. But $U_{\alpha + 2\beta} \leq \<U_\alpha, U_\beta\>$ centralises $S_\mu$, showing $S_\mu = 0$. This is a contradiction to $\mu \in M$.
\item
Finally, the case $\<\mu, \beta^\vee\> = 1$ was already dealt with in the previous configuration.
\qedhere
\end{itemize}
\end{proofclaim}

We have already proved inconsistency of a configuration of type $G_2$ (Remark \ref{r:G2trivial}).

%

This completes the proof of Proposition \ref{p:transitions}.
\end{proof}

As a consequence (and this was not obvious a priori), the Weyl group does act on the set of masses $M \subseteq E$; in particular for $\mu \in M$ and $\alpha, \beta \in \Phi_s$, one has $\<\sigma_\alpha(\mu), \beta^\vee\> \in \{-1, 0, 1\}$.
Therefore if $\mu \in M$ and $\alpha, \beta \in \Phi_s$ are adjacent in the Dynkin diagram, one cannot have $\<\mu, \alpha^\vee\> = \<\mu, \beta^\vee\> = 1$. (Notice that the proof we just gave did remove such configurations.)

\subsection{Intermezzo -- Isotypical Summands}

\begin{notation}\
\begin{itemize}
\item
Let $\mu \in M$ and $\cl(\mu)$ be the orbit of $\mu$ under the action of the Weyl group of $\bG$;
\item
let $V_{\cl(\mu)} = \oplus_{\nu \in \cl(\mu)} S_\nu$.
\end{itemize}
\end{notation}

\begin{corollary}\label{c:Vcl}
$V_{\cl(\mu)}$ is $\cG$-invariant.
\end{corollary}
\begin{proof}
It suffices to prove invariance under all maps $\d_{\pm \alpha, \lambda}$ for $(\alpha, \lambda) \in \Phi_s \times \K$. So let $\nu \in \cl(\mu)$ and $v\in S_\nu$.
\begin{itemize}
\item
If $\<\nu, \alpha^\vee\> = 0$ then $S_\nu \leq \Triv_V(\cG_\alpha)$ is annihilated by $\d_{\pm \alpha, \lambda}$.
\item
Now suppose $\<\nu, \alpha^\vee\> = 1$. Then $S_\nu \leq \lfloor\cU_\alpha, V\rfloor$ is annihilated by $\d_{\alpha, \lambda}$, which is obviously linear.
Recall from Proposition \ref{p:localanalysis} that in $\End(V)$, $\d_{-\alpha, \lambda} = \omega_\alpha \d_{\alpha, \lambda} \omega_\alpha^{-1}$.
As a consequence,
\[
\d_{-\alpha, \lambda} v = \omega_\alpha \d_{\alpha, \lambda} \omega_\alpha^{-1} v = - \omega_\alpha \d_{\alpha, \lambda} \omega_\alpha v = \omega_\alpha \tau_{\alpha, \lambda} v \quad \in \quad \omega_\alpha S_\nu = S_{\sigma_\alpha(\nu)} \leq V_{\cl(\mu)}
\]
\item
There is a similar argument if $\<\nu, \alpha^\vee\> = -1$.
\qedhere
\end{itemize}
\end{proof}

\subsection{Linear Structure}

By Corollary \ref{c:Vcl} we may suppose $V = V_{\cl(\mu_0)}$ for some $\mu_0 \in M$; if $\mu_0 = 0$ then $\cl(\mu_0) = \{0\}$ and $V = \Triv_V(\cG)$: we are done. So we may suppose $\mu_0 \neq 0$.

\begin{notation}\label{n:alpha0}\
\begin{itemize}
\item
Let $\alpha_0 \in \Phi_s$ with $\<\mu_0, \alpha_0^\vee\> = 1$ (up to taking $\sigma_{\alpha_0}(\mu_0)$ instead of $\mu_0$ there is one such).
\item
For $\gamma = (\alpha_1, \dots, \alpha_d) \in \Phi_s^d$, let $\sigma_\gamma = \sigma_{\beta_d} \circ \dots \circ \sigma_{\alpha_1}$ and $\omega_\gamma = \omega_{\alpha_d} \dots \omega_{\alpha_1} \in \End(V)$.

(Be careful that despite the notation, $\sigma_\gamma$ need not be a reflection.)
\end{itemize}
\end{notation}

We now define a field action piecewise on the various spots. Notice that whenever $\sigma_\gamma(\mu) = \nu$, then by Proposition \ref{p:transitions}, $\omega_\gamma$ restricts to a group isomorphism $S_\mu \to S_\nu$.

\begin{notation}\label{n:action}
Let $\lambda \in \K$ and $v \in S_\mu$ for some $\mu \in \cl(\mu_0)$. Take $\gamma \in \Phi_s^d$ with $\sigma_\gamma(\mu_0) = \mu$ and define:
\[\lambda\cdot v = \omega_\gamma \tau_{\alpha_0, \lambda} \omega_\gamma^{-1} v \]
\end{notation}

\begin{remark}
To be more specific, the field appearing here is such that $\cG_{\alpha_0}$ is $\mathbb{L}$-split.
\end{remark}

\begin{proposition}\label{p:linear}
This turns $V$ into a $\K[\cG]$-module.
\end{proposition}
\begin{proof}

Here again we make a series of claims.


\begin{claim}\label{p:linear:c:welldefined}
Notation \ref{n:action} is well-defined.
\end{claim}
\begin{proofclaim}
By the definition of $\cl(\mu_0)$ and since the reflections $\sigma_\alpha$ ($\alpha \in \Phi_s$) generate the Weyl group, there is at least one sequence $\gamma \in \Phi_s^d$ with $\sigma_\gamma(\mu_0) = \nu$. The problem is that the actual operator $\omega_\gamma$ may depend on $\gamma$: the basic example is $\sigma_{\alpha_0}^2(\mu_0) = \mu_0$, whereas $\omega_{\alpha_0}^2$ acts on $S_{\mu_0}$ as $-1$.

It suffices to show the following: if $\gamma, \gamma'$ are sequences such that $\sigma_\gamma(\mu_0) = \sigma_{\gamma'}(\mu_0)$, then there is $\varepsilon \in \{\pm 1\}$ with $(\omega_\gamma)_{|S_{\mu_0}} = \varepsilon (\omega_{\gamma'})_{|S_{\mu_0}}$.
Notice by inspection that $(\omega_\alpha^{-1})_{S_\mu}$ equals $\pm (\omega_\alpha)_{S_\mu}$ (the sign is given by $(-1)^{\<\mu, \alpha^\vee\>}$ as one can see), so we may replace any $\omega_\alpha$ by its inverse in a product of type $\omega_\gamma$.

So applying $\omega_{\gamma'}^{-1}$ it therefore suffices to prove: if $\sigma_\gamma(\mu_0) = \mu_0$ then $\omega_\gamma$ acts as $\pm 1$ on $S_{\mu_0}$.
(We may have missed something as this looks decently obvious but we failed to convey this feeling and have no better reason to offer the reader than the following argument.)

Write $\gamma = (\alpha_1, \dots, \alpha_d)$; for $i \in \{1, \dots, d\}$ let $\mu_i = \sigma_{\alpha_i}(\mu_{i-1})$. We suppose $\mu_d = \mu_0$ and shall prove that there is $\varepsilon \in \{\pm 1\}$ such that for any $v \in S_{\mu_0}$, one has $\omega_\gamma v = \varepsilon v$ (be careful that $\varepsilon$ will depend on both $\gamma$ and $\mu_0$). The proof will be by induction on $d$.
For convenience let $k_i = \<\mu_{i-1}, \alpha_i^\vee\> \in \{-1, 0, 1\}$; by definition, $\mu_i = \mu_{i-1} - k_i \alpha_i$.

First suppose that there is $i \in \{1, \dots, d\}$ with $k_i = 0$. Let $\gamma' = (\alpha_1, \dots, \widehat{\alpha_i}, \dots, \alpha_d)$ (i.e., remove $\alpha_i$ from the sequence). By assumption, $\mu_i = \mu_{i-1}$; hence $\sigma_{\gamma'}(\mu_0) = \sigma_\gamma(\mu_0) = \mu_0$. Also recall that $k_i = \<\mu_{i-1}, \alpha_i^\vee\> = 0$ implies that $S_{\mu_{i-1}} \leq \Triv_V(\cG_{\alpha_i})$: hence $\omega_{\alpha_i}$ fixes $S_{\mu_{i-1}}$ pointwise. So $\omega_\gamma v = \omega_{\gamma'} v$ and we may apply induction to conclude.

Now suppose there is $i \in \{1, \dots, d-1\}$ with $k_{i+1} = - k_i$. The left-hand side is:
\[k_{i+1} = \<\mu_i, \alpha_{i+1}^\vee\> = \<\sigma_{\alpha_i}(\mu_{i-1}), \alpha_{i+1}^\vee\> = \<\mu_{i-1}, \alpha_{i+1}^\vee\> - k_i \<\alpha_i, \alpha_{i+1}^\vee\>\]
Hence $\<\mu_{i-1}, \alpha_{i+1}^\vee\> = k_i (\<\alpha_i, \alpha_{i+1}^\vee\> - 1) \in \{-1, 0, 1\}$.
\begin{itemize}
\item
If $\<\alpha_i, \alpha_{i+1}^\vee\> = 2$ then $\alpha_{i+1} = \alpha_i$. Let $\gamma' = (\alpha_1, \dots, \widehat{\alpha_i}, \widehat{\alpha_{i+1}}, \dots, \alpha_d)$; clearly $\sigma_{\gamma'}(\mu_0) = \mu_0$ and $\omega_\gamma v = - \omega_{\gamma'} v$; apply induction.
\item
Otherwise $\<\alpha_i, \alpha_{i+1}^\vee\> \leq 0$ and this forces $\<\alpha_i, \alpha_{i+1}^\vee\> = 0$: the roots are not adjacent, implying that $\sigma_{\alpha_i}$ and $\sigma_{\alpha_{i+1}}$ on the one hand, $\omega_{\alpha_i}$ and $\omega_{\alpha_{i+1}}$ on the other hand, commute. So swapping these roots in the sequence, $\gamma' = (\alpha_1, \dots, \alpha_{i+1}, \alpha_i, \dots, \alpha_d)$ enjoys both $\sigma_{\gamma'}(\mu_0) = \sigma_\gamma(\mu_0)$ and $\omega_{\gamma'} v = \omega_\gamma v$. (The careful reader will note that $\mu_i$ changes, but $\mu_i$ is a mere gadget in our argument.)
\end{itemize}

Inductively applying the previous, we may suppose that there is $\ell \leq d$ with $k_1 = \dots = k_\ell = - k_{\ell +1} = \dots = - k_d$. Now $\mu_0 = \sigma_\gamma(\mu_0) = \mu_0 + k_1 (\alpha_1 + \dots + \alpha_\ell - \alpha_{\ell+1} \dots - \alpha_d)$. Since simple roots are linearly independent in the vector space they span, there is $i \leq \ell$ maximal with $\alpha_i = \alpha_{\ell+1}$. But like above, we see that $\alpha_{\ell+1}$ is never adjacent to $\alpha_j$ for $j \in \{i+1, \dots, \ell\}$. In particular $\gamma' = (\alpha_1, \dots, \alpha_i, \alpha_{l+1}, \alpha_{i+1}, \dots, \alpha_\ell, \alpha_{\ell + 2}, \dots, \alpha_d)$ (obtained from $\gamma$ by moving $\alpha_{\ell+1}$ right after $\alpha_i$) enjoys both $\sigma_{\gamma'}(\mu_0) = \mu_0$ and $\omega_{\gamma'} v = \omega_\gamma v$.
Now $\gamma'$ bears a redundancy; conclude by induction.
\end{proofclaim}


\begin{claim}
Notation \ref{n:action} defines a field action.
\end{claim}
\begin{proofclaim}
We argue piecewise; it clearly suffices to prove the claim in the action on $S_{\mu_0}$. Additivity in $v$ is obvious, so we now fix $v \in S_{\mu_0}$. Since $\alpha_0$ is the only root involved in the argument, we shall conveniently let $\alpha = \alpha_0$.

If $\cG = \fg$, then additivity in $\lambda$ is obvious since $\tau_{\alpha, \lambda} = h_{\alpha, \lambda}$; we turn to multiplicativity. Observe how, since $v \in S_{\mu_0} \leq V_{(\mu_0, \alpha^\vee)} = \fu_\alpha \cdot V$:
\[\lambda \cdot v = h_{\alpha, \lambda} v = x_{-\alpha} x_{\alpha, \lambda} v - x_{\alpha, \lambda} x_{-\alpha} v = - x_{\alpha, \lambda} x_{-\alpha} v = - x_\alpha x_{- \alpha, \lambda} v\]
so that, using quadraticity of $\fu_{\alpha}$ and looking if necessary at Remark \ref{r:y}:
\begin{align*}
\lambda (\lambda' v) & = x_{\alpha} x_{-\alpha, \lambda} x_{\alpha, \lambda'} x_{-\alpha} v\\
& = x_{\alpha} h_{\alpha, \lambda \lambda'} x_{-\alpha} v\\
& = - 2 x_{\alpha, \lambda \lambda'} x_{-\alpha} v + h_{\alpha, \lambda \lambda'} x_{\alpha} x_{-\alpha} v\\
& = 2 h_{\alpha, \lambda \lambda'} v - h_{\alpha, \lambda \lambda'} v\\
&= (\lambda \lambda') v
\end{align*}
as desired.

If $\cG = G$, then multiplicativity in $\lambda$ is now obvious since $\tau_{\alpha, \lambda} = t_{\alpha, \lambda}$; we turn to additivity. But remember from Proposition \ref{p:localanalysis} that $\d_{\alpha, \lambda} w_{\alpha} v = - t_{\alpha, \lambda} v$, so that, using quadraticity of $U_{\alpha}$:
\begin{align*}
(\lambda + \lambda') v & = - \d_{\alpha, \lambda + \lambda'} w_{\alpha} v\\
& = - (u_{\alpha, \lambda + \lambda'} - 1) w_{\alpha} v\\
& = - (u_{\alpha, \lambda} u_{\alpha, \lambda'} - 1) w_{\alpha} v\\
& = - (\d_{\alpha, \lambda} + \d_{\alpha, \lambda'} + \d_{\alpha, \lambda} \d_{\alpha, \lambda'}) w_{\alpha} v\\
& = - \d_{\alpha, \lambda} w_{\alpha} v - \d_{\alpha, \lambda'} w_{\alpha} v\\
& = \lambda v + \lambda' v
\end{align*}
as desired.
\end{proofclaim}


\begin{claim}
The action of $\cG$ on the $\K$-vector space $V$ is linear.
\end{claim}
\begin{proofclaim}
Remark that all operators $\omega_\beta$ for $\beta \in \Phi_s$ are linear by construction (and well-definedness of the action).

It could be tempting to prove linearity of one root $\bSL_2$-substructure, say $\cG_{\alpha_0}$, and of the Weyl group. The problem is that properly speaking, the Weyl group (the group of automorphisms \emph{of the root system} generated by $\{\sigma_\beta: \beta \in \Phi_s\}$) does not act on $V$. Of course we just observed that $\omega_\beta$ does act linearly; the problem remains to see why the image of $\cG$ in $\End(V)$ is generated by $\cG_{\alpha_0}$ and the operators $\{\omega_\beta: \beta \in \Phi_s\}$. This is obvious in the case of the group but not entirely so in the case of the Lie ring. So we take a side approach.

We shall first prove that all operators $\tau_{\beta, \lambda}$ for $(\beta, \lambda)\in \Phi_s \times \K^\times$ are linear.
Notice that since $h_{-\beta, \lambda} = - h_{\beta, \lambda}$ and $t_{-\beta, \lambda} = t_{\beta, \lambda}^{-1}$ (see our realisation), this will actually imply linearity of $\tau_{\pm \beta, \lambda}$.

In the case of the group $\cG = G$, assuming $\nu = \sigma_\gamma(\mu_0)$ and letting $v \in S_\nu$:
\begin{align*}
\tau_{\beta, \lambda} (\lambda' \cdot v) & = \tau_{\beta, \lambda} \omega_\gamma \tau_{\alpha, \lambda'}\omega_\gamma^{-1} v\\
& = \omega_\gamma \tau_{\alpha, \lambda'}\omega_\gamma^{-1} \tau_{\beta, \lambda} v\\
& = \lambda' \cdot (\tau_{\beta, \lambda} v)
\end{align*}
since $\omega_\gamma \tau_{\alpha, \lambda'}\omega_\gamma^{-1} \in \cT \leq C_G(\tau_{\beta, \lambda})$.

In the case of the Lie ring $\cG = \fg$ remember from Proposition \ref{p:localanalysis} that in $\End(V)$ the operators $\omega_\alpha$ (and therefore operators $\omega_\gamma$ as well) normalise the image of the abelian ring $\ft$. So we can carry exactly the same argument.
Hence $\cT$ acts linearly in any case.

We can now deduce that all elements $\d_{\pm \beta, \lambda}$ for $(\beta, \lambda) \in \Phi_s \times \K_+$ are linear. This will suffice for the linearity of $\cG$.
Fix $\nu \in \cl(\mu_0)$ and $v \in S_\nu$; also take $\lambda' \in \K$.
We show that $\d_{\pm \beta, \lambda} (\lambda' \cdot v) = \lambda' \cdot \d_{\pm \beta, \lambda} v$.
If $\<\nu, \beta^\vee\> = 0$ there is nothing to prove. By symmetry we may assume $\<\nu, \beta^\vee\> = -1$. Then $\d_{-\beta, \lambda}$ acts as the zero map on $S_\nu$ and therefore is linear. Now $\omega_\beta S_\nu = S_{\sigma_\beta(\nu)} \leq [U_\beta, V]$ so for any $v \in S_\nu$ one has:
\[\d_{\beta, \lambda} v = - \d_{\beta, \lambda} \omega_\beta^2 v = \tau_{\beta, \lambda} \omega_\beta v\]
In particular,
\[\d_{\beta, \lambda} (\lambda' v) =
\tau_{\beta, \lambda} \omega_\beta (\lambda' v) =
\lambda' \cdot \tau_{\beta, \lambda} \omega_\beta v =
\lambda' \cdot \d_{\beta, \lambda} v
\]
which proves linearity of $\d_{\beta, \lambda}$.
\end{proofclaim}

This completes the proof of Proposition \ref{p:linear}.
\end{proof}

\begin{remark}
The linear structure may seem to depend on both $\mu_0$ and $\alpha_0$. It actually depends on neither. This can be seen as a consequence of the postlude.
\end{remark}

\subsection{Postlude}

So far we have turned every $V_{\cl(\mu_0)}$ with $\mu_0 \neq 0$ into a $\K[\cG]$-module, which could easily be proved a direct sum of irreducible $\K[\cG]$-modules where every root acts quadratically. In order to conclude to identification with a minuscule module it suffices to determine the weights involved. This we do without invoking \cite[Chap. VIII, \S7.3]{BLie78}, as we promised that the present work would be self-contained.

\begin{proposition}
$\cl(\mu_0)$ is one of the orbits obtained from a geometrically minuscule module.
\end{proposition}

\begin{proof}

It suffices to show that $\mu_0$ is a minuscule weight.

\begin{claim}
We may suppose that for all $\beta \in \Phi_s$, $\<\mu_0, \beta^\vee\> \geq 0$.
\end{claim}
\begin{proofclaim}
This is because the topological closure of the positive chamber is a fundamental domain for the action of $W$ on $E$ \cite[Chap. V, \S3.3, Th\'{e}or\`{e}me 2]{BLie456}.
\end{proofclaim}

Remember that a consequence of Proposition \ref{p:transitions} is that if $\mu \in M$ and $\alpha, \beta \in \Phi_s$ are adjacent in the Dynkin diagram, one cannot have $\<\mu, \alpha^\vee\> = \<\mu, \beta^\vee\> = 1$. This will be used repeatedly in the argument.

\begin{claim}
There is exactly one $\alpha \in \Phi_s$ with $\<\mu_0, \alpha^\vee\> = 1$.
\end{claim}
\begin{proofclaim}
Suppose that there are a segment $\Sigma$ of the Dynkin diagram and a mass $\mu$ with both $\forall \gamma \in \Sigma$, $\<\mu, \gamma^\vee\> \geq 0$ and two distinct $\alpha, \beta \in \Sigma$ with $\<\mu, \alpha^\vee\> = \<\mu, \beta^\vee\> = 1$. We may suppose the distance between $\alpha$ and $\beta$ to be minimal.

Notice that by Proposition \ref{p:transitions}, $\alpha$ and $\beta$ are not adjacent. Let $\gamma$ be the neighbour of $\alpha$ in $[\alpha \beta]$; by assumption, $\<\mu, \gamma^\vee\> \geq 0$; by Proposition \ref{p:transitions} again one cannot have $\<\mu, \gamma^\vee\> = 1$, so $\<\mu, \gamma^\vee\> = 0$.

Let $\nu = \sigma_\alpha(\mu)$; clearly $\nu$ takes non-negative values on $[\gamma \beta]$ and $\<\nu, \gamma^\vee\> = \<\nu, \beta^\vee\> = 1$, against minimality of $[\alpha \beta]$.
\end{proofclaim}

Let $\alpha_0$ be the unique simple root with $\<\mu_0, \alpha_0^\vee\> = 1$.
We shall draw Dynkin diagrams and label each simple root $\alpha$ with the value $\<\mu, \alpha^\vee\>$.

In case $\bG = A_n$, there is nothing to prove; let us first handle types $B_n$ and $C_n$.

\begin{claim}\label{p:lattice:c:doublebond}
If the Dynkin diagram contains a double bond, then $\alpha_0$ is the extremal short root.
\end{claim}
\begin{proofclaim}
Notice that by Proposition \ref{p:transitions}, the following is inconsistent for any mass $\mu$:
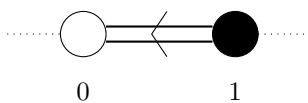
\[\begin{tikzpicture}
\draw (0,0) circle (.3);
\draw[fill=black] (2,0) circle (.3);
\node[anchor=south] at (0,-1) {$0$};
\node[anchor=south] at (2,-1) {$1$};
\draw[thick] (.3, .1) -- +(1.4,0);
\draw[thick] (.3, -.1) -- +(1.4,0);
\draw (1.1, .3) -- (.9, 0) -- (1.1, -.3);
\draw[dotted] (-1, 0) -- (-.3, 0);
\draw[dotted] (2.3, 0) -- (3, 0);
\end{tikzpicture}
\]
Therefore, inductively reflecting along the coroot with value $1$, the following is inconsistent as well:
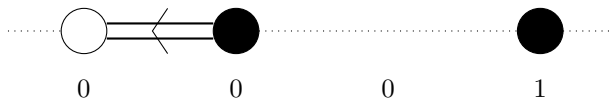
\[\begin{tikzpicture}
\draw (0,0) circle (.3);
\draw[fill = black] (2,0) circle (.3);
\draw[fill = black] (6,0) circle (.3);
\node[anchor=south] at (0,-1) {$0$};
\node[anchor=south] at (2,-1) {$0$};
\node[anchor=south] at (4,-1) {$0$};
\node[anchor=south] at (6,-1) {$1$};
\draw[thick] (.3, .1) -- (1.7,.1);
\draw[thick] (.3, -.1) -- (1.7,-.1);
\draw (1.1, .3) -- (.9, 0) -- (1.1, -.3);
\draw[dotted] (2.3, 0) to (5.7,0) ;
\draw[dotted] (-1, 0) -- (-.3, 0);
\draw[dotted] (6.3, 0) -- (7, 0);
\end{tikzpicture}
\]

On the other hand, reflecting in the middle then in the left root, the following is inconsistent too:
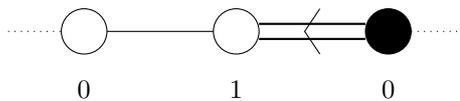
\[\begin{tikzpicture}
\draw (-2,0) circle (.3);
\draw (0,0) circle (.3);
\draw[fill=black] (2,0) circle (.3);
\node[anchor=south] at (-2,-1) {$0$};
\node[anchor=south] at (0,-1) {$1$};
\node[anchor=south] at (2,-1) {$0$};
\draw[thick] (.3, .1) -- +(1.4,0);
\draw[thick] (.3, -.1) -- +(1.4,0);
\draw (1.1, .3) -- (.9, 0) -- (1.1, -.3);
\draw[dotted] (-3, 0) -- (-2.3, 0);
\draw (-1.7, 0) -- (-0.3, 0);
\draw[dotted] (2.3, 0) -- (3, 0);
\end{tikzpicture}
\]
Inductively reflecting in the next-to-left then in the left coroot, so is the following:
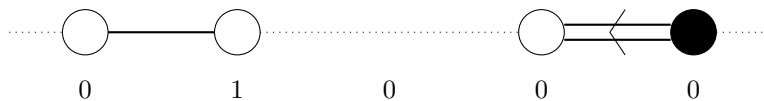
\[\begin{tikzpicture}
\draw (-6, 0) circle (.3);
\node[anchor=south] at (-6,-1) {$0$};
\draw (-4, 0) circle (.3);
\node[anchor=south] at (-4,-1) {$1$};
\draw (0,0) circle (.3);
\draw[fill=black] (2,0) circle (.3);
\node[anchor=south] at (-2,-1) {$0$};
\node[anchor=south] at (0,-1) {$0$};
\node[anchor=south] at (2,-1) {$0$};
\draw[thick] (-5.7, 0) -- +(1.4,0);
\draw[dotted] (-3.7, 0) -- +(3.4,0);
\draw[thick] (.3, .1) -- +(1.4,0);
\draw[thick] (.3, -.1) -- +(1.4,0);
\draw (1.1, .3) -- (.9, 0) -- (1.1, -.3);
\draw[dotted] (-7, 0) -- (-6.3, 0);
\draw[dotted] (2.3, 0) -- (3, 0);
\end{tikzpicture}
\qedhere
\]
\end{proofclaim}

In particular this covers the cases of $B_n$ and $C_n$.
We move to types $D_n$ and $E_n$.

\begin{claim}\label{p:lattice:c:DE}
If $\bG = D_n$ or $E_n$ then $\alpha_0$ is extremal.
\end{claim}
\begin{proofclaim}
The following is easily seen inconsistent:
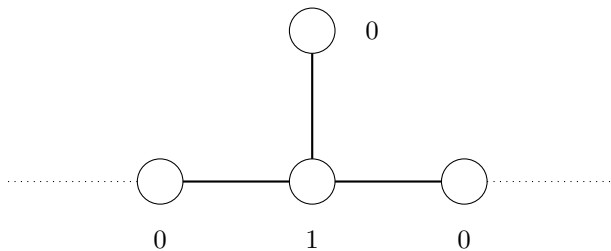
\[\begin{tikzpicture}
\draw (-2, 0) circle (.3);
\node[anchor=south] at (-2,-1) {$0$};
\draw (0, 0) circle (.3);
\node[anchor=south] at (0,-1) {$1$};
\draw (2, 0) circle (.3);
\node[anchor=south] at (2,-1) {$0$};
\draw (0, 2) circle (.3);
\node[anchor=east] at (1,2) {$0$};
\draw[thick] (-1.7, 0) -- (-0.3,0);
\draw[thick] (0.3, 0) -- (1.7,0);
\draw[thick] (0, 0.3) -- (0,1.7);
\draw[dotted] (-4, 0) -- (-2.3,0);
\draw[dotted] (2.3, 0) -- (4,0);
\end{tikzpicture}
\]
Therefore so is the following:
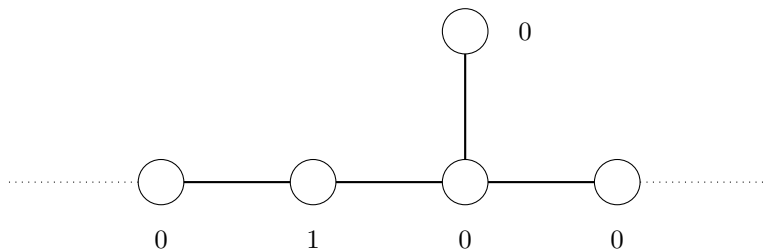
\[\begin{tikzpicture}
\draw (-4, 0) circle (.3);
\node[anchor=south] at (-4,-1) {$0$};
\draw (-2, 0) circle (.3);
\node[anchor=south] at (-2,-1) {$1$};
\draw (0, 0) circle (.3);
\node[anchor=south] at (0,-1) {$0$};
\draw (2, 0) circle (.3);
\node[anchor=south] at (2,-1) {$0$};
\draw (0, 2) circle (.3);
\node[anchor=east] at (1,2) {$0$};
\draw[thick] (-3.7, 0) -- (-2.3,0);
\draw[thick] (-1.7, 0) -- (-0.3,0);
\draw[thick] (0.3, 0) -- (1.7,0);
\draw[thick] (0, 0.3) -- (0,1.7);
\draw[dotted] (-6, 0) -- (-4.3,0);
\draw[dotted] (2.3, 0) -- (4,0);
\end{tikzpicture}
\]
By induction so is the following:
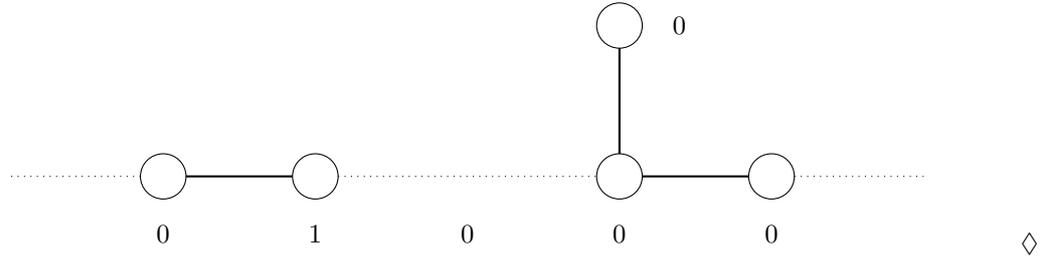
\[\begin{tikzpicture}
\draw (-6, 0) circle (.3);
\node[anchor=south] at (-6,-1) {$0$};
\draw (-4, 0) circle (.3);
\node[anchor=south] at (-4,-1) {$1$};
\node[anchor=south] at (-2,-1) {$0$};
\draw (0, 0) circle (.3);
\node[anchor=south] at (0,-1) {$0$};
\draw (2, 0) circle (.3);
\node[anchor=south] at (2,-1) {$0$};
\draw (0, 2) circle (.3);
\node[anchor=east] at (1,2) {$0$};
\draw[dotted] (-8, 0) -- (-6.3,0);
\draw[dotted] (-3.7, 0) -- (-0.3,0);
\draw[thick] (0.3, 0) -- (1.7,0);
\draw[thick] (0, 0.3) -- (0,1.7);
\draw[thick] (-5.7, 0) -- (-4.3,0);
\draw[dotted] (2.3, 0) -- (4,0);
\end{tikzpicture}
\qedhere\]
\end{proofclaim}

This covers case $D_n$. We are not done with case $E_n$.

\begin{claim}
If $\bG = E_n$ then $n = 6$ or $7$ and $\alpha_0$ is one of the roots (resp. the root) further from the arity $3$ root.
\end{claim}
\begin{proofclaim}
We know from Claim \ref{p:lattice:c:DE} that $\alpha_0$ is extremal but there remains a number of configurations to kill.

First, we shall check the following is inconsistent:
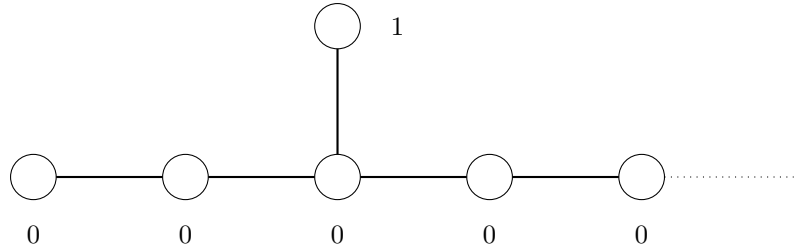
\[\begin{tikzpicture}
\draw[dotted] (4.3, 0) -- (6,0);
\foreach \k in {0,1,2,3,4}
  \draw (-4+2*\k, 0) circle (.3);
\foreach \k in {0,1,2,3,4}
 \node[anchor=south] at (-4+2*\k,-1) {$0$};
\foreach \k in {0,1,2,3}
\draw[thick] (-4+2*\k+.3, 0) -- +(1.4,0);
\draw (0,2) circle (.3);
 \node[anchor=east] at (1,2) {$1$};
\draw[thick] (0, .3) -- +(0,1.4);
  \end{tikzpicture}
\]
We see this by bringing the diagram into the following state:
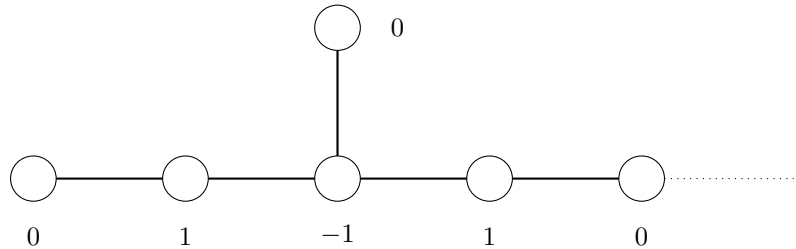
\[\begin{tikzpicture}
\draw[dotted] (4.3, 0) -- (6,0);
\foreach \k in {0,1,2,3,4}
  \draw (-4+2*\k, 0) circle (.3);
\node[anchor=south] at (-4,-1) {$0$};
\node[anchor=south] at (-2,-1) {$1$};
\node[anchor=south] at (0,-1) {$-1$};
\node[anchor=south] at (2,-1) {$1$};
\node[anchor=south] at (4,-1) {$0$};
\foreach \k in {0,1,2,3}
\draw[thick] (-3.7+2*\k, 0) -- +(1.4,0);
\draw (0,2) circle (.3);
 \node[anchor=east] at (1,2) {$0$};
\draw[thick] (0, .3) -- +(0,1.4);
  \end{tikzpicture}
\]
Then into:
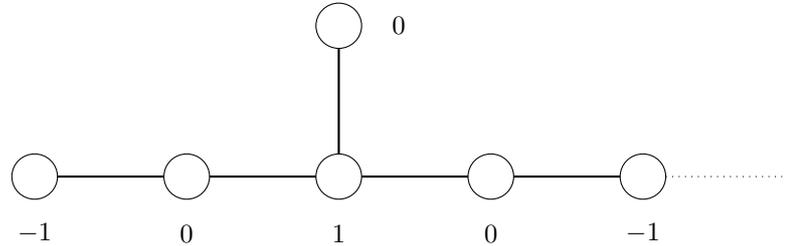
\[\begin{tikzpicture}
\draw[dotted] (4.3, 0) -- (6,0);
\foreach \k in {0,1,2,3,4}
  \draw (-4+2*\k, 0) circle (.3);
\node[anchor=south] at (-4,-1) {$-1$};
\node[anchor=south] at (-2,-1) {$0$};
\node[anchor=south] at (0,-1) {$1$};
\node[anchor=south] at (2,-1) {$0$};
\node[anchor=south] at (4,-1) {$-1$};
\foreach \k in {0,1,2,3}
\draw[thick] (-4+2*\k+.3, 0) -- +(1.4,0);
\draw (0,2) circle (.3);
 \node[anchor=east] at (1,2) {$0$};
\draw[thick] (0, .3) -- +(0,1.4);
  \end{tikzpicture}
\]
an inconsistent configuration as we know from the proof of Claim \ref{p:lattice:c:DE}.

The counting reader will find three more configurations to kill: one for $E_7$, two for $E_8$. We can remove two simultaneously.
Perhaps we ought to make our notations more compact. Consider the diagram:
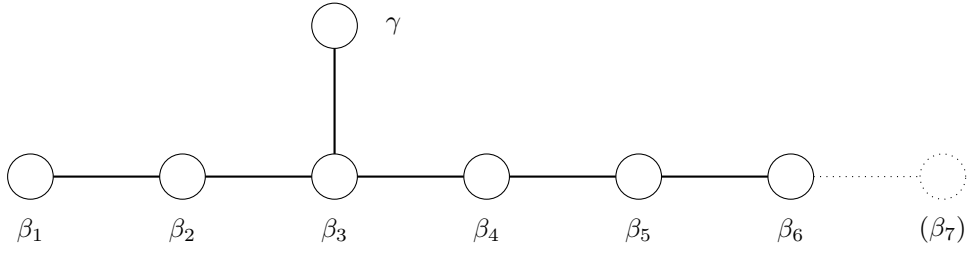
\[\begin{tikzpicture}
\foreach \k in {1,2,3,4,5,6}
  \draw (2*\k, 0) circle (.3);
\draw[dotted] (14, 0) circle (.3);
\foreach \k in {1,2,3,4,5,6}
 \node[anchor=south] at (2*\k,-1) {$\beta_{\k}$};
 \node[anchor=south] at (14,-1) {$(\beta_7)$};
\foreach \k in {1,2,3,4,5}
\draw[thick] (2*\k+.3, 0) -- +(1.4,0);
\draw[dotted] (12.3, 0) -- +(1.4,0);
\draw (6,2) circle (.3);
 \node[anchor=east] at (7,2) {$\gamma$};
\draw[thick] (6, .3) -- +(0,1.4);
  \end{tikzpicture}
\]
We tabulate consecutive masses until we reach inconsistency (an empty cell is an unchanged value):
\[\begin{array}{c|ccccccc}
\gamma & \beta_1 & \beta_2 & \beta_3 & \beta_4 & \beta_5 & \beta_6 & (\beta_7)\\\hline
0 & 1 & 0 & 0 & 0 & 0 & 0\\
 & -1 & 1\\
& 0 & -1 & 1\\
1 & & 0 & -1 & 1\\
& & & 0 & -1 & 1\\
& & & & 0 & -1 & 1\\
-1 & & & 1\\
0 & & 1 & -1 & 1\\
& & & 0 & -1 & 0\\
& 1 & -1 & 1\\
1 & & 0 & -1 & 0\\
- 1 & & & 0\\
\end{array}
\]
In the final state, the value at $\beta_i^\vee$ for $i \in\{1, \dots, 6\}$ is non-negative, and positive at both $\beta_1^\vee$ and $\beta_6^\vee$: an inconsistency.

So there remains only one $E_8$ configuration, which we handle as follows.
\[\begin{array}{c|ccccccc}
  \gamma & \beta_1 & \beta_2 & \beta_3 & \beta_4 & \beta_5 & \beta_6 & \beta_7\\\hline
0 & 0 & 0 & 0 & 0 & 0 & 0 & 1\\
& & & & & & 1 & -1\\
& & & & & 1 & -1 & 0\\
& & & & 1 & -1 & 0\\
& & & 1 & -1 & 0\\
1 & & 1 & -1 & 0 & & & \\
-1 & & & 0\\
& 1 & -1 & 1\\
 & -1 & 0\\
0 & & 1 & -1 & 1\\
& 0 & -1 & 0\\
& & & 1 & -1 & 1\\
1 & & 0 & -1 & 0\\
-1 & & & 0\\
& & & & 1 & -1 & 1\\
& & & 1 & -1 & 0\\
0 & & 1 & -1 & 0\\
& 1 & -1 & 0\\
& -1 & 0\\
& & & & & 1 & -1 & 1\\
& & & & 1 & -1 & 0\\
& & & 1 & -1 & 0\\
1 & & 1 & -1 & 0\\
-1 & & & 0
\end{array}
\]
and $[\beta_2 \beta_7]$ is an inconsistent configuration.
\end{proofclaim}

\begin{claim}
For $\bG = F_4$ the configuration is inconsistent.
\end{claim}
\begin{proofclaim}
By Claim \ref{p:lattice:c:doublebond} only the following need be considered:
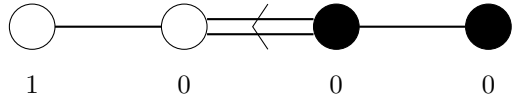
\[\begin{tikzpicture}
\draw (-2,0) circle (.3);
\draw (0,0) circle (.3);
\draw[fill = black] (2,0) circle (.3);
\draw[fill = black] (4,0) circle (.3);
\node[anchor=south] at (-2,-1) {$1$};
\node[anchor=south] at (0,-1) {$0$};
\node[anchor=south] at (2,-1) {$0$};
\node[anchor=south] at (4,-1) {$0$};
\draw[thick] (-1.7, 0) -- (-.3,0);
\draw[thick] (.3, .1) -- (1.7,.1);
\draw[thick] (.3, -.1) -- (1.7,-.1);
\draw (1.1, .3) -- (.9, 0) -- (1.1, -.3);
\draw[thick] (2.3, 0) -- (3.7,0);
\end{tikzpicture}
\]
We leave it to the reader to push the configuration to inconsistency.
\end{proofclaim}

\begin{claim}
For $\bG = G_2$ the configuration is inconsistent.
\end{claim}
\begin{proofclaim}
By Proposition \ref{p:transitions}, $\alpha_0$ cannot be the long simple root (call it $\beta$) and is therefore the short root; reflecting in $\alpha_0^\vee$ we find $\<\sigma_{\alpha_0}(\mu_0), \beta^\vee\> = - \<\alpha_0, \beta^\vee\> = 1$, and then $\<\sigma_\beta \sigma_{\alpha_0}(\mu_0), \alpha_0^\vee\> = - 1 + 3 \notin\{-1, 0, 1\}$: a contradiction.
\end{proofclaim}

This shows that $\mu_0$ is one of the minuscule weights described in \cite[Chap. VIII, end of \S7.3]{BLie78}.
\end{proof}

This immediately gives an isomorphism of $\K[\cG]$-modules: so $V_{\cl(\mu)}$ is a sum of minuscule representations, in the geometric sense of the term.

\medskip\hrule\medskip
Future variations will see our return to model theory: we shall untensor a cubic $\SL_2(\K)$-module in the finite Morley rank category.

\bibliographystyle{plain}
\bibliography{../English/Variationen}

\begin{thebibliography}{10}

\bibitem{BLie456}
Nicolas Bourbaki.
\newblock {\em \'{E}l\'ements de math\'ematique. {F}asc. {XXXIV}. {G}roupes et
  alg\`ebres de {L}ie. {C}hapitre {IV}: {G}roupes de {C}oxeter et syst\`emes de
  {T}its. {C}hapitre {V}: {G}roupes engendr\'es par des r\'eflexions.
  {C}hapitre {VI}: syst\`emes de racines}.
\newblock Actualit\'es Scientifiques et Industrielles, No. 1337. Hermann,
  Paris, 1968.

\bibitem{BLie78}
Nicolas Bourbaki.
\newblock {\em \'{E}l\'ements de math\'ematique. {F}asc. {XXXVIII}: {G}roupes
  et alg\`ebres de {L}ie. {C}hapitre {VII}: {S}ous-alg\`ebres de {C}artan,
  \'el\'ements r\'eguliers. {C}hapitre {VIII}: {A}lg\`ebres de {L}ie
  semi-simples d\'eploy\'ees}.
\newblock Actualit\'es Scientifiques et Industrielles, No. 1364. Hermann,
  Paris, 1975.

\bibitem{CSimple}
Roger~W. Carter.
\newblock {\em Simple groups of {L}ie type}.
\newblock Wiley Classics Library. John Wiley \& Sons, Inc., New York, 1989.
\newblock Reprint of the 1972 original, A Wiley-Interscience Publication.

\bibitem{TV-I}
Adrien Deloro.
\newblock Ver\"anderungen \"uber einen {S}atz von {T}immesfeld -- {I}.
  {Q}uadratic actions.
\newblock {\em Confluentes Math.}, 5(2):23--41, 2013.

\bibitem{TV-II}
Adrien Deloro.
\newblock Symmetric powers of {N}at $\mathfrak{sl}_2(\mathbb{K})$.
\newblock {\em Communications in Algebra}, 2015.
\newblock To appear.

\bibitem{TV-III}
Adrien Deloro.
\newblock Symmetric powers of {N}at $\mathrm{SL}_2(\mathbb{K})$.
\newblock {\em Journal of Group Theory}, 2015.
\newblock To appear.

\bibitem{SGA3}
Philippe Gille and Patrick Polo, editors.
\newblock {\em Sch\'emas en groupes ({SGA} 3). {T}ome {III}. {S}tructure des
  sch\'emas en groupes r\'eductifs}.
\newblock Documents Math\'ematiques (Paris) [Mathematical Documents (Paris)],
  8. Soci\'et\'e Math\'ematique de France, Paris, 2011.
\newblock S{\'e}minaire de G{\'e}om{\'e}trie Alg{\'e}brique du Bois Marie
  1962--64. [Algebraic Geometry Seminar of Bois Marie 1962--64], A seminar
  directed by M. Demazure and A. Grothendieck with the collaboration of M.
  Artin, J.-E. Bertin, P. Gabriel, M. Raynaud and J-P. Serre, Revised and
  annotated edition of the 1970 French original.

\bibitem{Smith}
Stephen~D. Smith.
\newblock Quadratic action and the natural module for {${\rm SL}_2(k)$}.
\newblock {\em J. Algebra}, 127(1):155--162, 1989.

\bibitem{TIdentification}
Franz~Georg Timmesfeld.
\newblock On the identification of natural modules for symplectic and linear
  groups defined over arbitrary fields.
\newblock {\em Geom. Dedicata}, 35(1-3):127--142, 1990.

\bibitem{Timmesfeld}
Franz~Georg Timmesfeld.
\newblock {\em Abstract root subgroups and simple groups of {L}ie type},
  volume~95 of {\em Monographs in Mathematics}.
\newblock Birkh\"auser Verlag, Basel, 2001.

\bibitem{TComplete}
Franz~Georg Timmesfeld.
\newblock Complete reducibility of quadratic modules for finite {L}ie-type
  groups.
\newblock {\em J. Algebra}, 355:35--60, 2012.

\end{thebibliography}

\end{document}